\documentclass{article}
\usepackage[T1]{fontenc}
\usepackage[utf8]{inputenc}
\usepackage[english]{babel}
\usepackage{graphicx, color}
\usepackage{amssymb, amsmath, amsfonts}
\usepackage{dsfont}
\usepackage{vmargin}
\usepackage{amsthm} 

\everymath{\displaystyle}

\usepackage{hyperref}
\hypersetup{
	colorlinks=true,
	breaklinks=true,
	urlcolor=blue,
	linkcolor=red,
	bookmarksopen=true,
	pdftitle = Stochastic Navier-Stokes-Cahn-Hilliard
	pdfauthor = Ludovic Gouden\`ege and Luigi Manca,
	pdfsubject = Stochastic Navier-Stokes-Cahn-Hilliard
	pdfkeywords = {Navier-Stokes, Cahn-Hilliard, stochastic partial differential equations}
}

\numberwithin{equation}{section} 
\renewcommand{\thefootnote}{\*}

\newcommand{\R}{\mathbb{R}} 
\newcommand{\N}{\mathbb{N}} 
\newcommand{\E}{\mathbb{E}} 
\newcommand{\sL}{\mathrm{L}}
\newcommand{\sH}{\mathrm{H}}
\newcommand{\sV}{\mathrm{V}}
\newcommand{\sK}{\mathrm{K}}
\newcommand{\dd}{\mathrm{d}} 
\newcommand{\dds}{\mathrm{d}s}
\newcommand{\ddt}{\mathrm{d}t}

\newcommand{\ddx}{\mathrm{d}x}

\newcommand{\Div}{\mathrm{div\,}} 
\newcommand{\qspace}{Q} 
\newcommand{\Prj}{\mathcal{P}} 
\newcommand{\Prb}{\mathbb{P}} 
\newcommand{\Tr}{\mathrm{Tr}} 
\newcommand{\la}{\left\langle}
\newcommand{\ra}{\right\rangle}
\newcommand{\lp}{\left(}
\newcommand{\rp}{\right)}
\newcommand{\Cp}{C_{\mathcal{P}}} 
\newcommand{\Ca}{\Sigma} 
\newcommand{\Cb}{\Xi} 

\newtheorem{Th}{Theorem}[section]
\newtheorem{Def}[Th]{Definition}
\newtheorem{Coro}[Th]{Corollary}
\newtheorem{Le}[Th]{Lemma}
\newtheorem{Prop}[Th]{Proposition}
\newtheorem{Rem}[Th]{Remark}
\newtheorem{Hyp}[Th]{Hypothesis}

\title{Stochastic phase field $\alpha$-Navier-Stokes vesicle-fluid interaction model.}
\author{Ludovic Gouden\`ege\footnotemark[1] \ and Luigi Manca\footnotemark[2]}
\date{}

\begin{document}

\maketitle
\renewcommand{\thefootnote}{\*}
\footnotetext[0]{\!\!\!\!\!\!\!\!\!\!\!AMS 2000 subject classifications. 60H15, 60H30, 37L55, 35Q30, 35Q35, 76D05}
\renewcommand{\thefootnote}{\fnsymbol{footnote}}
\footnotetext[1]{CNRS, Fédération de Mathématiques de CentraleSupélec FR 3487, CentraleSupélec, 91190 Gif-sur-Yvette, France, goudenege@math.cnrs.fr}
\footnotetext[2]{LAMA, Université Paris-Est - Marne-la-Vallée, 77454 Marne-la-Vall\'ee, France, luigi.manca@u-pem.fr}

\begin{abstract}
We consider a stochastic perturbation of the phase field alpha-Navier-Stokes model with vesicle-fluid interaction.
It consists in a system of nonlinear evolution partial differential equations modeling the fluid-structure interaction associated to the dynamics of an elastic vesicle immersed in a moving incompressible viscous fluid. This system of equations couples a phase-field equation -for the interface between the fluid and the vesicle- to the alpha-Navier-Stokes equation -for the viscous fluid- with an extra nonlinear interaction term, namely the bending energy.

The stochastic perturbation is an additive space-time noise of trace class on each equation of the system. We prove the existence and uniqueness of solution in classical spaces of $\sL^{2}$ functions with estimates of non-linear terms and bending energy. It is based on a priori estimate about the regularity of solutions of finite dimensional systems, and tightness of the approximated solution.
\end{abstract}

\newpage

\section*{Introduction and main results}

This paper is devoted to study a random perturbation of the equations governing the dynamic of an elastic vesicle immersed in a moving incompressible viscous fluid, whose deterministic model have been studied in \cite{DuLiLiu} and \cite{EntringerBoldrini}.\\

According to \cite{Seifert}, these equations are key research in the study of the dynamics of cells in fluid media.
This type of models are of crucial importance in biology, where the analysis of the deformation of vesicles immersed in fluids is central topic.
In particular we can refer to the articles on the biological aspects (see \cite{AbkarianLartigueViallat,BeauncourtRioualSionBibenMisbah,BibenKassnerMisbah,DuLiuRyhanWang:05a,DuLiuRyhanWang:05b, DuLiuRyhanWang:05c, LiuTakahashiTucsnak}). 
In all these articles there is a common idea about usefulness of phase field approaches.
The phase field approaches, compared to sharp interface models, are natural ways to include several important physical aspects of the phenomenon being considered, without complexity of the free-boundary value problems, both in the theoretical and numerical aspects.\\

First consider the $\alpha$-Navier-Stokes equation which reads, on the time interval $[0,T]$, on smooth, open and bounded space domain $\qspace\subset \R^{N}$ in dimension $N=2$ or $3$, with $\nu$ the constant viscosity, and $\rho$ the constant density of the incompressible fluid:
\begin{equation} \label{eq.aNS}
\begin{cases}
\partial_t u+ (w \cdot \nabla) u + (\nabla w)^{T} \cdot u+\frac{1}{\rho}\nabla p = f +\nu \Delta u ,\\
u=w - \alpha^{2}\Delta w - \nabla q,\\
\Div (u)=\Div(w)=0.
\end{cases}
\end{equation}
where $\Delta$ is the Laplace operator and $f$ is the forcing.
The unknowns\footnote {With $u=(u_{1},u_{2},u_{3})$ in dimension $N=3$ or $u=(u_1,u_2)$ in dimension $N=2$.} are the random fields $p$ and $u$ (also $q$ and $w$), which respectively represent the (modified\footnote{Here the modified or hydrodynamic pressure satisfies $p=\pi - \frac{\rho}{2} |w|^{2}$ where $\pi$ is the  pressure.})
pressure and the averaged velocity vector field of the point $x$ at time $t$. Both unknowns $u$ and $w$ have homogeneous Dirichlet boundary conditions, and the pressures $p$ and $q$ are defined up to an additive term which could be used to stay divergence free.
This model takes part of a general class of regularized models for high Reynolds number flows, firstly proposed by Leray in \cite{Leray:34a,Leray:34b} for Euler equations.
Some authors stress that, from the biological point of view,
the $\alpha$-Navier-Stokes type equations are relevant since they are adequate for flows with high Reynolds number (like in turbulence),
which may occur in some biological situations.
This model is also known as viscous Camassa-Holm or the Lagrangian Averaged Navier-Stokes-$\alpha$ (LANS-$\alpha$) model.
These models have been introduced by Holm, Marsden and Ratiu in \cite{HolmMarsdenRatiu:98a,HolmMarsdenRatiu:98b}. 
It has been studied in the deterministic case by Foias, Holm and Titi (see \cite{FoiasHolmTiti:01,FoiasHolmTiti:02})
which have obtained the necessary estimations about the non-linear term in the Navier-Stokes equation in periodic domain.
There are also works in alternative conditions about domain and boundary conditions in \cite{BjorlandSchonbek, Caglar, GuermondOdenPrudhomme}.
See also \cite{MarsdenShkoller} for the link between Camassa-Holm and LANS-$\alpha$ models.\\

In \cite{DuLiLiu} and \cite{EntringerBoldrini} the authors have considered $\alpha$-Navier-Stokes model for the fluid coupled with a phase field equation for the membrane of the vesicle. They have introduced a forcing term $f$ which is a non-linear additive term depending of the phase field unknown.
The form of the interaction is given by the variational derivative of a bending energy of the membrane of the vesicle.
We obtain a system of interaction in the space-time domain $[0,T]\times \qspace$ between the fluid and the phase-field equations under the form : 
\begin{equation} \label{eq.aNSPF}
 \begin{cases}
\partial_t (w+\alpha^2Aw)+\nu A(w+\alpha^2 \Delta w)+ \widetilde{B}(w,w+\alpha^2 A w)
=\Prj\left(\dfrac{\delta E(\phi)}{\delta \phi}\nabla \phi\right),\\
  \Div (w)=0, \\
 \phi_t+w\cdot \nabla \phi = -\gamma \dfrac{\delta E(\phi)}{\delta \phi},
 \end{cases}
\end{equation}
where $\phi$ is the phase field unknown/order parameter which describes the membrane of the vesicle, with the linear Stokes operator $A=-\Prj\Delta$, the Leray orthogonal projector $\Prj$ on divergence free space $\sH$, and the non-linear operator $\widetilde{B}$ which will be described later.

This unknown $\phi$ takes the values $+1$ outside the membrane and $-1$ inside, with a thin transition width characterized by a small positive parameter~$\varepsilon$. The surface of the membrane corresponds to the points where $\phi=0$, which is actually a very complex area described by the level-set approach, but not explicitly considered in the phase field approach, or in various numerical approaches. The term $\dfrac{\delta E(\phi)}{\delta \phi}$ is sometimes called the {\em chemical potential}. It is multiplied by the constant $\gamma$ which is a positive real number controlling the strength of the chemical potential. Moreover this term can be modeled using various description, depending of the physical consideration about the vesicle. 

It is assumed that the energy associated with the deformation of the vesicle membrane comes mainly from the bending energy. Actually this energy is not directly well adapted to {\em a priori} estimate of quantities related to $\phi$ (like its norm in Sobolev spaces) since the vesicle tends to minimize the quantity
\[
f_{\varepsilon}(\phi)=-\varepsilon\Delta\phi-\frac{\phi}\varepsilon(1-\phi^2),
\]
by minimization of the penalized bending energy given by
\[
 \mathcal{E}_\varepsilon(\phi)=\frac{k}{2\varepsilon}\int_\qspace\lp \varepsilon\Delta \phi+\dfrac\phi\varepsilon\lp1-\phi^ 2\rp\rp^2\ddx = \frac{k}{2\varepsilon}\int_{\qspace} f_{\varepsilon}(\phi)^{2} \ddx,
\]
with the physical parameter $k$ of low relevance here. This bending energy is clearly not a norm or the sum of two competitive behaviors like in classical Allen-Cahn or Cahn-Hilliard equations. Although it implies only a second-order differential operator, this energy is more close to a fourth-order differential linearity as in the Cahn-Hilliard model. Thus the difficulty of the model comes from this form of energy.
Moreover, knowing that the volume and the surface area of the vesicle are basically preserved, we penalize the bending energy $\mathcal{E}_\varepsilon$  by adding extra terms to form the total energy $E$ : 
\[
 E(\phi)=  \mathcal{E}_\varepsilon(\phi)+\frac12 M_1(\mathcal{A}(\phi)-a)^2+\frac12 M_2(\mathcal{B}_{\varepsilon}(\phi)- b)^2,
\]
where
\[
 \mathcal{A}(\phi)=\int_\qspace\phi\  \ddx,
\]
\[
  \mathcal{B}_{\varepsilon}(\phi)=\int_\qspace\left(\dfrac{\varepsilon}{2}|\nabla\phi|^2+\dfrac{1}{4\varepsilon}(\phi^2-1)^2\right)\ddx,
\]
with $M_1$, $M_2$ which are (large) constants used to enforce that the volume and the surface area of the vesicle remain the same. The constants $a$ and $b$ are physical parameters related to the actual volume and surface area of the vesicle
(see \cite{DuLiuWang:04} for details).

Finally -and this is the novelty in the modeling- we assume that there exist two stochastic perturbations $\xi_{w}$ and $\xi_{\phi}$ which are the derivative of space-time noises $W$ and $Z$, thus formally $\xi_{w} = dW$ and $\xi_{\phi}=dZ$. These perturbations are added linearly to both equations
of the system of interaction via covariance operators $\Ca$ and $\Cb$.
\begin{Hyp} \label{hyp.trace}
We assume that
\[ 
  \Tr[\Ca^*\Ca]<\infty,\qquad \Tr[\Cb^*\Delta^2 \Cb]<\infty.
\]

\end{Hyp} 

From a physical perspective, the stochastic perturbation can be seen as an unknown internal microscopic thermal agitation, or a random source. The technical assumptions of the noises permit to use the It\^o-formula, which is the key to obtain {\em a priori} estimates depending of the trace of the operators. We obtain the abstract formulation of our studied system
\begin{equation}\label{eq.aNSCHabstract}
 \begin{cases}
 d(w+\alpha^2 A w)=\left(-\nu A(w+\alpha^2 A w)- \widetilde{B}(w,w+\alpha^2 A w) +\Prj\left(\dfrac{\delta E(\phi)}{\delta \phi}\nabla \phi\right) \right)dt + \Ca dW_{t}\\
 d\phi=\left(-w\cdot \nabla \phi  -\gamma \dfrac{\delta E(\phi)}{\delta \phi}\right)dt+\Cb dZ_{t}.
 \end{cases}
\end{equation}

Moreover this system is endowed by boundary and initial conditions
\begin{equation*} 
\begin{cases}
   w=0, \quad A w =0,& \text{on } [0,T]\times\partial \qspace,\\
   \phi=-1,\quad \Delta \phi =0, &  \text{on } [0,T]\times \partial \qspace,\\
   u(0,x) = u_{0}(x) & \text{on } \qspace,\\
   \phi(0,x) = \phi_{0}(x) & \text{on } \qspace,\\
\end{cases} 
\end{equation*}
with initial data $u_{0}$ and $\phi_{0}$.
\begin{Rem}
The apparently extra boundary condition $A w=0$ makes sense,
since in $\alpha$-Navier-Stokes model, we study a couple of unknowns $w$ and $u=w +\alpha^{2}A w$ (the pressure $q$ disappears with Leray's projection) which have both homogeneous Dirichlet boundary condition $u=w=0$ on $\partial \qspace$. Thus $\alpha^{2}A w =0$ on $\partial \qspace$.
\end{Rem}

This system is composed of two stochastic partial differential equations which are coupled by an energy. So this is clear that the results obtained in this paper about existence and uniqueness of solution can be extended to more general forms of coupling energy, as soon as it permits a control of some norm of $\phi$ in Hilbert space with space regularity. Actually the studied form of energy is a mixing between fourth-order Cahn-Hilliard equation and second-order Allen-Cahn equation.
These types of stochastic equations with additive noise have been studied in many works. For the Cahn-Hilliard equation there are results about existence and uniqueness in \cite{CardonWeber,DaPratoDebussche,ElezovicMikelic} with polynomial nonlinearity, and in \cite{DebusscheGoudenege,Goudenege,GoudenegeManca:16} for singular nonlinearity and space-time white noises, or degenerate noises. We can also cite a result of existence for a stochastic partial differential equation with a mixing between Cahn-Hiliard and Allen-Cahn equation with multiplicative noise. It has been obtained in \cite{AntonopoulouKaraliMillet} with estimations on the Green functions in the spirit of \cite{Walsh,Gyongy}. Concerning the stochastic Navier-Stokes equation, we can cite the important work present in \cite{HairerMattingly:06,HairerMattingly:11, Odasso}. Using approximated equations in finite dimensional space, we have exhibited a priori estimates and compactness of a sequence of solution of these approximated equations. It permits to prove existence (and uniqueness) of weak (martingale) solution obtained by convergence in weak topology of classical spaces $\sL^{2}(\qspace)$. Precisely we have proved the following:
\begin{Th}\label{thm.intro}\ \newline
\indent Let $T>0$ and $(w_0,\phi_0)\in D(A)\times \sL^2(\qspace)$ with $\phi_{0}=-1$ on $\partial \qspace$.

Assume that the linear operators $(\Ca,\Cb)$ satisfy Hypothesis \ref{hyp.trace}.

Then there exists a unique weak solution $((w,\phi), (\Omega, \mathcal{F}, \mathbb{P}, (\mathcal{F}_{t})_{t\in[0,T]}),(W,Z))    $ 

of problem \eqref{eq.aNSCHabstract}.
Moreover, for any $k\in \N^*$ there exists a constant $c=c(k,T,w_0, \phi_0)>0$ such that
\begin{equation} \label{eq.aNSCHestimate}
 \begin{split}
  &\E\left[\sup_{0\leq t\leq T}\left(  \|w(t)\|_{\sH}  +|\phi(t)|_{2}\right)^k  \right] \leq c,\\
  &\E\left[\int_0^{T} \left( |w|_2^2 +\alpha^2 |\nabla w|_2^2+E(\phi)\right)^{k-1} \left(\nu(|\nabla w|_2^2 +\alpha^2|A  w|_2^2 )
     +\gamma\left|\dfrac{\delta E(\phi)}{\delta \phi}\right|_2^2\right) \dd s\right]\leq c.
   \end{split}
\end{equation}
Finally, $\phi,w$ are continuous in mean square, that is for any $t_0\geq0$ we have
\[
  \lim_{t\to t_0}\E\left[|w(t)-w(t_0)|_V^2  \right]=0,\qquad \lim_{t\to t_0}\E\left[|\phi(t)-\phi(t_0)|_{\sL^2(\qspace)}^2  \right]=0.
\]
\end{Th}

In section 1, we will describe notations about spaces, classical inequalities and nonlinear estimates about the bending energy which are of crucial importance for the proof of the main result. Moreover we describe the definition of a solution of equation \eqref{eq.aNSCHabstract}. In section 2, we derive a priori estimate and we prove technical lemmas which will be used in the proof of the main theorem. Finally in section 3, under the hypotheses of \ref{thm.intro}, we prove existence and uniqueness of solution which satisfies \ref{thm.intro}. This result is a corollary of a more general result obtained in Section 3 about existence and uniqueness of solution with an approximation procedure in finite dimensional spaces. In particular we prove continuity of solution with respect to time with values in Sobolev spaces, and $\sL^{p}$ integrability of solution with respect to time with values in Sobolev spaces ($\sH^{1}$ for fluid unknown $w$ and $\sH^{4}$ for parameter order $\phi$).

\section{Spaces, inequalities and nonlinear estimates}

The $\alpha$-Navier-Stokes equation \eqref{eq.aNS} can be formulated in the equivalent form given in \eqref{eq.aNSPF}. We need to explain this equivalence, since this is the core of the variational formulation. First we introduce the following spaces:
\begin{itemize}
 \item $\mathcal{C}_{0}^{\infty}(\qspace)$ is the space of infinitely differentiable functions with compact support;
 \item $\sL^{2}(\qspace)$, $\sL^{p}(\qspace)$, $\sH^k(\qspace)$, $\sH^k_0(\qspace)$, $W^{p,k}(\qspace)$ denotes the usual Sobolev spaces for integrability order $p\in\N$ and derivative order $k\in\N$; when the functions are vector-valued in dimension $N=1,2,3$, we write $(\sL^2(\qspace))^N$;
 \item $(\cdot,\cdot)$ denotes the inner product of the Hilbert space $(\sL^2(\qspace))^N$, with $N=1,2,3$;
 \item $|\cdot|_{p}$ denotes the norm in the space $(\sL^{p}(\qspace))^{N}$, with $p\in\N$ and $N=1,2,3$;
 \item $\|\cdot\|_{\sL}$ denotes the norm in a generic space $\sL$;
 \item $\la\cdot,\cdot\ra_{\sL',\sL}$ denotes the duality between a generic space $\sL$ and its dual space $\sL'$;
 \item $X$ is the space $(\sH^{1}_{0}(\qspace))^{N} \cap (\sH^{2}(\qspace))^{N}$;
 \item $\sH$ is the closure in $(\sL^2(\qspace))^{N}$ of $\{u\in X:\Div(u)=0\}$;
 \item $\sV$ is the closure in $(\sH^1(\qspace))^{N}$ of $\{u\in X:\Div(u)=0\}$;
\end{itemize}

\subsection{The Stokes operator $A$}

We denote by $\Prj:(\sL^2(\qspace))^N\to \sH$ the Leray orthogonal projector.
The Stokes operator is then defined by
\[
   A:=-\Prj\Delta : D(A)\to \sH,
\]
with domain $D(A)=X \cap V\subset \sH$.
The operator $A$ is self adjoint and positive. 
Its inverse, $A^{-1}:\sH\to \sH$, is a compact self adjoint operator, thus $\sH$ admits an orthonormal basis $\{e_{j}\}_{j\in\N^{*}}$ formed by the eigenfunctions of $A$, i.e. $A e_{j} = \lambda_{j}e_{j}$, with $0 < \lambda_{1}\leq \lambda_{2}\leq \cdots \leq \lambda_{j} \rightarrow+\infty$. For $\rho\in\R$, the Sobolev spaces $D(A^\rho)$ are the closure of $C_0^\infty(\qspace)$ with respect to the norm
\[
\|x\|_{D(A^\rho)}=\left(\sum_{k}(1+\lambda_k^{2\rho})\la x,e_k\ra^2  \right)^{\frac12}.
\]
As well known (see, for instance, \cite{FoiasHolmTiti:02}) the operator $A$ can be continuously extended to $V=D(A^{\frac12})$ with values in $V'=D(A^{-\frac12})$ such that for all $u$, $v \in \sV$
\[
\la Au,v\ra_{V',V}=(A^{1/2}u,A^{1/2}v)=\int_\qspace (\nabla u\cdot \nabla v)\ dx,
\]
Similarly $A^2$ can be continuously extended to $D(A)$ with values in $D(A)'$ (the dual space of the Hilbert space $D(A)$) such that
for all $u$, $v \in D(A)$
\[
\la A^2u,v\ra_{D(A)',D(A)}=(Au,Av).
\]
One can show that there is a constant $c>0$ such that for all $w \in D(A)$
\[
c^{-1} |Aw|_{2} \leq \|w \|_{\sH^{2}} \leq c\  |Aw|_{2}.
\]

This operator $A$ could also be used to define a stochastic convolution thanks to the $strongly$ continuous semigroup $(e^{tA})_{t\geq 0}$ by the formula
\[
W_A(t)=\int_0^t e^{(t-s)A}dW(s),
\]
for cylindrical Wiener processes, which could be used for instance to define mild solutions.
This is not the choice made here, since we have enough regularity to define solution with variational estimation.


\subsection{The bilinear form $\widetilde{B}$}


The specific form of $\alpha$-Navier-Stokes equation \eqref{eq.aNS} has been studied in \cite{Caglar} for bounded domains, or in \cite{FoiasHolmTiti:01} as the Kelvin-filtered Navier-Stokes equation. This equation is also known as the viscous version of the Camassa-Holm equation. It has been studied in 
\cite{ChenFoiasHolmOlsonTitiWynne} and for periodic domain in \cite{FoiasHolmTiti:02}.
But the global well-posedness for the Lagrangian averaged Navier-Stokes (LANS-$\alpha$) equations
on bounded domains have been studied in \cite{MarsdenShkoller} where the authors describe the equivalence between different formulations.
Precisely they show that the $\alpha$-Navier-Stokes equation \eqref{eq.aNS} is equivalent to LANS-$\alpha$ equations under the condition $Aw=0$ on $\partial \qspace$. We do not present all the details, but the central idea is to define a bilinear operator associated to the non-linear part of equation \eqref{eq.aNS} in the spirit of the usual bilinear operator $B(w,u)=\Prj\left[(w \cdot \nabla ) u\right]$ of Navier-Stokes equations.
It is well defined for all $w$, $u\in (\sH^{1}_{0}(\qspace))^{N}$, and such that for all $w$, $u$ and $v\in V\subset (\sH^{1}_{0}(\qspace))^{N}$
\[
\left(B(w,u),v\right) = -\left(B(w,v),u\right).
\]

Thus, applying $\Prj$ to the equation \eqref{eq.aNS} and using the identity
\[
(w\cdot \nabla)u+ (\nabla w)^{T} u=-w\times (\nabla\times u)+\nabla(u\cdot w),
\]
we can see that the nonlinear term of equation \eqref{eq.aNS} could be replaced by 
the bilinear operator $\widetilde{B}$ defined for all $w$, $u\in (\sH^{1}_{0}(\qspace))^{N}$ by 
\[
\widetilde{B}(w,u) =-\Prj[w\times (\nabla\times u)]
\]
since $\nabla(u\cdot w)$ is in the orthogonal of $V$. This operator appears clearly in the Camassa-Holm formulation.

The next results will be  crucial for many proofs:
\begin{Prop} \label{prop.1.1}\ \newline
\noindent 
(i) The operator $\widetilde{B}$ can be extended continuously to $\sV\times \sV$ with values in $\sV'$; for all $u,v,w\in \sV$ it satisfies
\begin{eqnarray*}
 &\left|\la \widetilde{B}(u,v),w\ra_{V',V}\right|\leq c\|u\|_{\sH}^{1/2}\|u\|_{\sV}^{1/2}\|v\|_{\sV}\|w\|_{\sV},\\
 &\left|\la \widetilde{B}(u,v),w\ra_{V',V}\right|\leq c\|u\|_{\sV}\|v\|_{\sV}\|w\|_{\sH}^{1/2}\|w\|_{\sV}^{1/2},\\
 & \la \widetilde{B}(u,v),w\ra_{V', V}=-\la \widetilde{B}(w,v),u\ra_{V',V}, \quad \text{ and } \quad  \la \widetilde{B}(u,v),u\ra_{V', V}=0.\\
\end{eqnarray*}
\noindent
(ii)
Its restriction to $D(A)$ satisfies for all $u\in V$, $v\in H$, $w\in D(A)$ it holds
\[
  \left|\la \widetilde{B}(u,v),w\ra_{D(A)',D(A)}\right|\leq  c  \|u\|_{\sV} \|v\|_{\sH} \|w\|_{D(A)}
\]
\end{Prop}
\begin{proof}
  The proof of {\em (i)} is classical and be found, for instance, on \cite{FoiasHolmTiti:02}.
  The statement  {\em (ii)} follows easily by the estimate 
  \[
  \left|\la \widetilde{B}(u,v),w\ra_{D(A)',D(A)}\right|\leq c\left(   \|u\|^{1/2}_{\sH} \|u\|^{1/2}_{\sV}\|v\|_{\sH} \|Aw\|_{\sH} + \|u\|_{\sV} \|v\|_{\sH}\|w\|_{\sV}^{1/2}\|Aw\|_{\sH}^{1/2}   \right)
\]
which can be found, for instance, in \cite{EntringerBoldrini}.
\end{proof}

\subsection{Definition of solutions} 

We are now able to define the concept of solution of equation \eqref{eq.aNSPF} or more precisely the solution of its abstract form \eqref{eq.aNSCHabstract}.

\begin{Def} \label{def.sol} Let $T>0$ and $(w_0,\phi_0)\in D(A)\times \sL^2(\qspace)$ with $\phi_{0}=-1$ on $\partial \qspace$.
Assume that the linear operators $(\Ca,\Cb)$ satisfy Hypothesis \ref{hyp.trace}. 
We say that  $((w,\phi), (\Omega, \mathcal{F}, \mathbb{P}, (\mathcal{F}_{t})_{t\in[0,T]}),(W,Z))$ is a weak solution of \eqref{eq.aNSCHabstract} if 
\begin{itemize}
\item  $(\Omega, \mathcal{F}, \mathbb{P}, (\mathcal{F}_{t})_{t\in[0,T]})$ is a complete filtered probability space.
\item $w$, $\phi$ are adapted to the filtation $(\mathcal{F}_{t})_{t\in[0,T]}$.
\end{itemize}
Moreover,  $\Prb$-a.s.,
\begin{itemize}
\item $w\in \sL^{2}([0,T];D(A))$; 
\item $w+\alpha^2Aw\in \mathcal{C}([0,T];D(A)'))\cap \sL^2([0,T];\sH)$;
\item $\tilde B(w,w+\alpha^2Aw)\in \sL^2([0,T];D(A)')$;
\item $\phi\in \sL^{2}([0,T];\sH^{2}(\qspace))\cap \mathcal{C}([0,T];\sL^2(\qspace))$ such that $\phi+1=\Delta \phi =0$ on $\partial \qspace$; 
\item $\frac{\delta E (\phi)}{\delta \phi} \in \sL^{2}([0,T];\sL^{2}(\qspace))$  (this term will be defined in Section \ref{section.nonlinear});
\item $w \cdot \nabla \phi \in \sL^{2}([0,T];\sL^{2}(\qspace))$;
\item $\frac{\delta E (\phi)}{\delta \phi}\nabla \phi \in  \sL^{2}([0,T];D(A)')$;
\item For all $\xi \in D(A)$,  for all $t\in [0,T]$ we have
\begin{equation}   \label{eq.def.sol}
 \begin{cases}   
 \la w(t)+\alpha^2 A w(t),\xi \ra= \la w_{0}+\alpha^2 A w_{0},\xi \ra -\nu  \int_{0}^{t}\la w+\alpha^2 A w, A\xi\ra \dd s\\
 \qquad  - \int_{0}^{t} \la \widetilde{B}(w,w+\alpha^2 A w),\xi\ra \dd s 
  +\int _{0}^{t}\la\dfrac{\delta E(\phi)}{\delta \phi}\nabla \phi,\xi \ra \dd s  +  \la \Ca^{*} \xi,  W_{t}\ra\\
 \phi(t)= \phi_{0} - \int_{0}^{t} \left(w\cdot \nabla \phi  +\gamma \dfrac{\delta E(\phi)}{\delta \phi}\right)\dd s + \Cb Z_{t}.
 \end{cases}
\end{equation}
\end{itemize}
\end{Def}

\subsection{Nonlinear estimates} \label{section.nonlinear}

From now and to the end of this article we will skip the parameters $\varepsilon$ and $k$ (set to the value $1$) because it does not bring very useful information and the reading will be clearly simplified. For this reason we will be very cautious about cancellation during subtraction of terms.

The variational derivative of $E$ with respect to the variable $\phi$ at point $\phi$ in the direction $\psi$ is defined for any $\phi+1,\psi\in \mathcal{C}_0^\infty(\qspace)$. 
by 
\begin{eqnarray*}
\la \dfrac{\delta E}{\delta\phi}(\phi),\psi\ra &=&\lim_{h\to0}\frac{E(\phi+h\psi)-E(\phi)}{h}=\int_\qspace \dfrac{\delta E}{\delta\phi}(\phi)\psi \ \ddx\\
&=& \int_\qspace f(\phi) \left(f'(\phi)\psi\right) \ \ddx+
  M_1(\mathcal{A}(\phi)-a)\mathcal{A}(\psi)+M_2(\mathcal{B}(\phi)-b)\int_\qspace f(\phi)\psi \ \ddx
\end{eqnarray*}
Here we have set 
\[
  f'(\phi)\psi=- \Delta\psi +(3\phi^2-1)\psi.
\]
In this case the variational derivative of $E$ can be identified with 
\begin{equation}\label{eq.varE2}
 \dfrac{\delta E}{\delta\phi}(\phi)
   =  \Delta^2\phi-\Delta(\phi^3-\phi) +(3\phi^2-1)f(\phi) + M_1(\mathcal{A}(\phi)-a)+M_2(\mathcal{B}(\phi)-b)f(\phi).
\end{equation}

\begin{Prop}
There exists $c>0$ such that for any $\phi+1\in C_0^\infty(\qspace)$ it holds 
\begin{equation} \label{eq.ineq.E2}
    |\Delta\phi|_2^2+|\nabla \phi|_2^4+|\phi\nabla\phi|_2^2+|\phi|_4^8+|\phi|_6^6\leq c(1+E(\phi)).
\end{equation}
and
\begin{equation} \label{eq.ineq.E2bis}
   E(\phi)\leq c(1+\|\phi\|_{\sH^2}^8) 
\end{equation}
\end{Prop}

\begin{proof}
 We have
  \begin{equation} \label{eq.termeB}
   4\mathcal{B}(\phi)=2|\nabla\phi|_2^2+|\phi^2-1|_2^2 =2|\nabla\phi|_2^2+|\phi|_4^4-2|\phi|_2^2+|\qspace|.
 \end{equation}
 Then, since $|\phi|_2^2\leq \frac12|\qspace|^2+\frac12|\phi|_4^4$, there exists $c>0$ such that
 \begin{equation}  \label{eq.bound.B}
  |\nabla\phi|_2^2+|\phi|_4^4+|\phi|_2^2\leq c(1+\mathcal{B}(\phi)).
 \end{equation}
Clearly, for some other constant $c>0$ it holds 
\[
   |\nabla\phi|_2^4+|\phi|_4^8+|\phi|_2^4\leq c(1+(\mathcal{B}(\phi))^2)\leq c(1+E(\phi)).
\]
At this point, it remains to bound the quantity $|\Delta \phi|_2^2+|\phi\nabla\phi|_2^2+|\phi|_6^6$.
Using the expression of $f(\phi)$ we get
 \begin{eqnarray}
   2|f(\phi)|_2^2 &=& |-\Delta\phi+\phi(\phi^2-1)|^2_2  \notag  \\
      &=& |\Delta\phi|_2^2-2\langle \Delta \phi,\phi(\phi^2-1)\rangle + |\phi(\phi^2-1)|_2^2  \notag  \\
      &=& |\Delta\phi|_2^2+2\langle \nabla \phi, 3\phi^2\nabla \phi-\nabla\phi\rangle +|\phi|_6^6-|\phi|_2^2  \notag  \\
      &=& |\Delta\phi|_2^2+6|\phi\nabla\phi|_2^2-2|\nabla\phi|_2^2 +|\phi|_6^6-|\phi|_2^2  \label{eq.termef}
 \end{eqnarray}
 Here we use the fact that $\phi(\phi^2-1)\in C_0^\infty(\qspace)$ in order to perform integration by parts.
 Thus, 
 \[
    |\Delta\phi|_2^2+|\phi\nabla\phi|_2^2+|\phi|_6^6\leq 2|f(\phi)|_2^2+2|\nabla\phi|_2^2+|\phi|_2^2
 \]
Using the estimate \eqref{eq.bound.B} and elementary inequalities, there exist constants $c,c'>0$ such that
\[
  |\Delta\phi|_2^2+|\phi\nabla\phi|_2^2+|\phi|_6^6\leq 2|f(\phi)|_2^2+c(1+\mathcal{B}(\phi))\leq  c'(1+E(\phi)).		
\]
Then, \eqref{eq.ineq.E2} follows easily.
Let us show \eqref{eq.ineq.E2bis}.
By \eqref{eq.termef} and the embedding $\sH^1\subset \sL^{6}(\qspace)$ we find that for some $c>0$, independent by $\phi$ it holds 
\[
   2|f(\phi)|_2^2\leq |\Delta\phi|_2^2+6|\phi\nabla\phi|_2^2+|\phi|_6^6
    \leq  |\Delta\phi|_2^2+6|\phi\nabla\phi|_2^2+c\|\phi\|_{\sH^1}^6
\]
Moreover, by Poincaré inequality 
\[
 |\phi\nabla\phi|_2^2\leq |\phi|_\infty^2|\nabla\phi|_2^2\leq  (|\phi+1|_\infty+1)^2|\nabla\phi|_2^2
    \leq   (c|\nabla\phi|_2+1)^2|\nabla\phi|_2^2 \leq 2(c|\nabla\phi|_2^2+1)|\nabla\phi|_2^2.
\]
We deduce that for some $c>0$ 
\[
  2|f(\phi)|_2^2 \leq |\Delta\phi|_2^2+2c|\nabla \phi|_2^4+2c|\nabla\phi|_2^2+c\|\phi\|_{\sH^1}^6 \leq c(1+\|\phi\|_{\sH^2}^6).
\]
The inequality $(\mathcal{A}(\phi)-a)^2\leq c(1+|\phi|_2^2)$, with $c=c(a)$ follows immediately. 
By \eqref{eq.termeB} and the embedding $\sH^1\subset \sL^4(\qspace)$ we get 
\[
   4\mathcal{B}(\phi)\leq 2|\nabla\phi|^2+|\phi|_4^4 +|\qspace|\leq c2|\nabla\phi|^2+\|\phi\|_{\sH^1}^4 +|\qspace|
\]
Then, for some $c>0$ it holds
\[
  (\mathcal{B}(\phi)-b)^2\leq c(1+\|\phi\|_{\sH^1}^8).
\]
Then, by the bounds obtained above, we deduce that \eqref{eq.ineq.E2bis} holds for some $c>0$ independent by $\phi$.
\end{proof}

\begin{Prop} \label{prop.ineq.E3}
 There exists a constant $c>0$ such that for any $\phi+1\in  C_0^\infty(\qspace)$ it holds
  \begin{equation}    \label{eq.ineq.E3}
   |\Delta^2\phi|_2\leq  \left|\frac{\delta E}{\delta \phi}(\phi)\right|_2 +c(1+E(\phi)^2)
 \end{equation}
\end{Prop}

\begin{proof}
By \eqref{eq.varE2} we have
\[ 
   |\Delta^2\phi|_2\leq \left|\frac{\delta E}{\delta \phi}(\phi)\right|_2+I_1+I_2+I_3+I_4,
\]
where
\begin{eqnarray*}
   I_1&=&  \left|(-\Delta)\left((\phi^2-1)\phi\right)\right|_2, \\
   I_2&=&  \left|(3\phi^2-1)f (\phi)\right|_2 ,\\
   I_3&=&M_1\left|(\mathcal{A}(\phi)-a)\right|_2,\\
   I_4&=&M_2\left|(\mathcal{B}(\phi)-b) f (\phi)\right|_2 .
\end{eqnarray*}

For $I_1$ we have
\[
  I_1=  6\phi|\nabla\phi|^2+3\phi^2\Delta \phi-\Delta\phi   
\]
Then by basic inequality we get
\begin{eqnarray*}
  I_1 &\leq&  6|\phi|_\infty|\nabla\phi|_4^2+3|\phi|_\infty^2|\Delta \phi|_2+|\Delta\phi|_2  \\
\end{eqnarray*}
The Poincaré inequality yields $|\phi|_\infty\leq |\phi+1|_\infty+1\leq \Cp |\nabla\phi|_2+1$
where $\Cp$ is the Poincaré constant.
Moreover, by the Sobolev embedding $H^1(\qspace)\subset L^4(\qspace)$ we get $|\nabla \phi|_4^2\leq c|\nabla\phi|_{H^1}^2$ for some constant $c>0$ independent by $\phi$.
Then, using  repeatedly the Young inequality we get that there exists $c>0$ such that 
\begin{eqnarray*}
  I_1  &\leq&  6c \left(\Cp |\nabla\phi|_2+1 \right) |\nabla\phi|_{H^1}^2+3 \left(\Cp |\nabla\phi|_2+1 \right)^2|\Delta \phi|_2+|\Delta\phi|_2   \\
	  &\leq&  c\left(1+ |\nabla\phi|_2^2  + |\nabla\phi|_{H^1}^4+|\nabla\phi|_2^4+|\Delta \phi|_2^2+|\Delta\phi|_2\right).
\end{eqnarray*}
Notice that $|\nabla \phi|_{H^1}\leq |\phi|_{H^2}\leq c(|\phi|_2+|\nabla\phi|_2+|\Delta\phi|_2)$ for some $c>0$ independent of $\phi$.
Then, still using Young inequality, there exists a constant $c_1>0$ such that
\[
  I_1\leq  c_1 \left(1+|\phi|^2_4+|\nabla\phi|_2^4 +|\Delta \phi|_2^4  \right).
\]
For $I_2$, using the expression of $f(\phi)$ and the Poincaré inequality $|\phi+1|_\infty\leq C_p|\nabla\phi|_2$ we obtain\\
 \begin{eqnarray*}
 I_2  &\leq&     \left(3|\phi|_\infty^2+1\right)|f(\phi)|_2 \\
     &\leq&  \left(3(|\phi+1|_\infty+1)^2+1\right)\left( \frac12|\Delta\phi|_2+ \frac14\left(|\phi|_6^3+|\phi|_2\right) \right) \\
     &\leq&  \left(3(C_p|\nabla\phi|_2+1)^2+1\right)\left( \frac12|\Delta\phi|_2+ \frac14\left(|\phi|_6^3+|\phi|_2\right) \right). 
\end{eqnarray*} 
By applying  the inequality  $(\alpha+\beta)^2\leq 2\alpha^2+2\beta^2$  repeatedly, we find that there exists a constant $c_2>0$ such that
\[
  I_2\leq  c_2 \left(|\Delta\phi|_2^2 +|\nabla\phi|_2^4+|\phi|_6^6 +  |\phi|_2^2+1\right).
\]
Clearly, for $I_3$ there exists a constant $c_3>0$ such that
\[
  I_3\leq c_4\left(|\phi|_2+1   \right)
\]
For $I_4$ we have, by the expression of $\mathcal{B}(\phi)$ and $f(\phi)$,
\begin{eqnarray*}
  I_4  &\leq&   M_2(\mathcal{B}(\phi)+b)|f (\phi)|_2\\
           &\leq&  \left(\frac{1}{2}|\nabla \phi|_2^2+\frac{1}{4 }(|\phi|_4^4+2|\phi|_2^2+|\qspace|)+b\right)\left(  \frac12|\Delta\phi|_2+ \frac14\left(|\phi|_6^3+|\phi|_2\right)\right)\\
\end{eqnarray*}
Using the inequality $(\alpha+\beta)^2\leq 2\alpha^2+2\beta^2$ repeatedly, it is easy to show that there exists a constant $c_4>0$ such that
\[
  I_4 \leq c_4\left(  |\phi|_4^8+ |\phi|_2^4+  |\phi|_6^6+  |\nabla \phi|_2^4 + |\Delta\phi|_2^2+ 1\right). 
\]
Taking into account the estimates on $I_1,\ldots,I_4$, by \eqref{eq.ineq.E2} we deduce that there exists $c>0$ such that
\[
 I_1+I_2+I_3+I_4\leq C(|\phi|_2^4+|\phi|_4^8+|\phi|_6^6+|\nabla \phi|_2^4+ |\Delta\phi|_2^4)\leq c(1+(E(\phi))^2).   \qedhere
 \]
\end{proof}

\begin{Prop} \label{prop.CEestimate}
There exists a constant $c>0$ such that  for any $\phi+1\in C_0^\infty(\qspace)$
\begin{equation}  \label{eq.CEestimate}
    \left|\Cb^*\frac{\delta E}{\delta \phi}(\phi)\right|_2 \leq c\left(1+|\phi|_4^8+|\phi|_6^3+|\nabla\phi|_2^4+|\Delta \phi|_2^2\right) 
\end{equation}
\end{Prop}

\begin{proof}
 By \eqref{eq.varE2} we have
\[
  \Cb^*\frac{\delta E}{\delta \phi}(\phi)= I_1+I_2+I_3+I_4+I_5,
\]
where ( we recall that $k=\varepsilon=1$ )
\begin{eqnarray*}
   I_1&=&  \Cb^*\Delta^2\phi, \\
   I_2&=&  \Cb^*(-\Delta)\left((\phi^2-1)\phi\right), \\
   I_3&=&  \Cb^*(3\phi^2-1)f (\phi) ,\\
   I_4&=&M_1\Cb^*(\mathcal{A}(\phi)-a),\\
   I_5&=&M_2\Cb^*(\mathcal{B}(\phi)-b) f (\phi) .
\end{eqnarray*}
By hypothesis \ref{hyp.trace} and remark \ref{rem.C}, there exists $c_1>0$ such that 
\[
 |I_1|_2\leq  c_1|\Delta\phi|_2.
\]
For $I_2$, still by hypothesis \ref{hyp.trace} and remark \ref{rem.C} there exists $c>0$ such that 
\[
  |I_2|_2\leq  c|(\phi^2-1)\phi|_2   
\]
Then by basic inequality we get, for some $c_2>0$,
\begin{eqnarray*}
  |I_2|_2 &\leq&  c_2(|\phi|_6^3+|\phi|_2).
\end{eqnarray*} 
Since $\Cb^*$ is a bounded linear operator, the terms $I_3$, $I_4$, $I_5$ can be estimated as done for Proposition  \ref{prop.ineq.E3} to get 
\[
  |I_3|_2\leq  c_3 \left(|\phi|_4^4 + |\Delta\phi|_2^2 + |\phi|_6^6+ |\phi|_2^2+1\right)
\] 
\[
  |I_4|_2\leq c_4\left(|\phi|_2+1   \right)
\] 
\[
  |I_5|_2 \leq c_5\left(   |\nabla \phi|_2^4 + |\phi|_4^8+ |\phi|_2^4+  |\Delta\phi|_2^2+ |\phi|_6^6+ 1\right). 
\]
for some constant $c_3$, $c_4$, $c_5$ independent by $\phi$.
Taking into account the estimates for  $I_i$, $i=1,\ldots,5$, the claim follows. 
\end{proof}


The second variational of $E$ in $\phi$ is a bilinear form on $C_0^\infty(\qspace)\otimes C_0^\infty(\qspace)$  and takes the form
\begin{eqnarray*}
 \left(\dfrac{\delta^2 E}{\delta\phi^2}(\phi)\right)(\psi,\rho)
  &=&  \int \left(f' (\phi)\psi\right)\left(f' (\phi)\rho\right)\ \ddx 
   + \int_\qspace f (\phi)\left(f'' (\phi)(\psi,\rho) \right)\ \ddx 
   + M_1\mathcal{A}(\psi)\mathcal{A}(\rho) \\
   && +M_2\left(\int_\qspace f (\phi)\rho\ \ddx\right)\left(\int_\qspace f(\phi)\psi \ \ddx\right) 
 +M_2( \mathcal{B}(\phi)-b)\int_\qspace\left ( f' (\phi)\rho\right)\psi \ \ddx
\end{eqnarray*}
where
\[
    f'' (\phi)(\psi,\rho)=  6\phi \psi\rho.
\]
When $\psi=\rho$ it takes the form
 \begin{eqnarray*} 
    \left(\dfrac{\delta^2 E}{\delta\phi^2}(\phi)\right)(\psi,\psi)
  &=&  \int \left(f' (\phi)\psi\right)^2 \ \ddx 
   + \int_\qspace f (\phi)\left(f'' (\phi)(\psi,\psi) \right)\ \ddx 
   + M_1\left(\mathcal{A}(\psi)\right)^2\\
   &&  +M_2\left(\int_\qspace f (\phi)\psi \ \ddx\right)^2  
   +M_2( \mathcal{B}(\phi)-b)\int_\qspace\left ( f' (\phi)\psi\right)\psi \ \ddx.
\end{eqnarray*}


\begin{Prop} \label{Le.trace}
 There exists a constant $c>0$ such that for any $\phi+1,\psi\in C_0^\infty(\qspace)$ it holds
 \[
   \left(\dfrac{\delta^2 E}{\delta\phi^2}(\phi)\right)(\psi,\psi)
     \leq c\left(|\Delta\phi|_2^2+|\nabla\phi|_2^4+|\phi|_4^2+1\right)         \left(|\psi|_2^2+|\nabla \psi|_{2}^2+|\Delta \psi|_{2}^2\right)
 \]
\end{Prop}
\begin{proof}
Let us write 
\[
 \dfrac{\delta^2 E(\phi_{})}{\delta \phi_{}^2}(\psi,\psi) = I_1+I_2+I_3+I_4+I_5,
\]
where
\begin{eqnarray*}
  I_1  &=&   \int \left(f' (\phi_{})\psi\right)^2 \ \ddx \\
  I_2 &=&  \int_\qspace f (\phi_{})\left(f'' (\phi_{})(\psi,\psi) \right)\ \ddx \\
  I_3 &=&  M_1\left(\mathcal{A}(\psi)\right)^2\\
  I_4  &=& M_2\left(\int_\qspace f (\phi_{})\psi \ \ddx\right)^2  \\
  I_5  &=& M_2( \mathcal{B}(\phi_{})-b)\int_\qspace\left ( f' (\phi_{})\psi\right)\psi \ \ddx.
\end{eqnarray*}
For $I_1$ we have 
\begin{eqnarray*}
  I_1   &=&   |\left(f' (\phi)\psi\right)|_2^2\\
        &=&  |-  \Delta\psi + (3\phi^2-1)\psi|_2^2\\
      &\leq&   \left( |\Delta \psi|_2 +  (3|\phi|_\infty^2+1) | \psi|_2 \right)^2
\end{eqnarray*}
Since $|\phi|_\infty\leq |\phi+1|_\infty+1$ and  by Poincar\'e inequality there exists a constant $\Cp>0$ such that $|\phi+1|_\infty\leq \Cp|\nabla\phi|_2$,
the right hand side is bounded by
\[
   \left( |\Delta \psi|_2+(3(\Cp |\nabla\phi|_2+1)+1) | \psi|_2\right)^2.
\]
Then it follows that there exists a constant $d_1>0$ such that 
\[
  I_1\leq d_1|\Delta \psi|_2^2+d_1\left(|\nabla\phi|_2^2+1\right)|\psi|_2^2.
\]
For $I_2$ we have, using H\"older inequality,
\begin{eqnarray*}
  I_2   &=& 6\int_\qspace\left(-\Delta \phi+  (\phi^2-1)\phi\right)\phi\psi^2\ \ddx\\
        &\leq&  6\left( \int_\qspace|\Delta \phi \phi|\ddx + \int_\qspace|\phi^2-1|\phi^2\ddx   \right)|\psi|_\infty^2 \\
        &\leq& 6\left(\frac{1}{2}|\Delta \phi|^2_2
                +\frac{1}{2}|\phi|_2^2 + |\phi|_4^4+ |\phi|_2^2\right)|\psi|_\infty^2.
\end{eqnarray*}
Then there exists $d_2$ such that
\[
  I_2\leq d_2\left( |\Delta \phi|^2_2+|\phi|_4^4 +|\phi|_2^2 \right)|\psi|_\infty^2
\]
The term $I_3$ is easily bounded by
\[
   I_3\leq M_1 |\qspace||\psi|_2^2 = d_3 |\psi|_2^2,
\]
where $d_3=M_1 |\qspace|$.
For $I_4$ we have
\[
 I_4\leq M_2|f (\phi)|_2^2|\psi|_2^2.
\]
It is easy to see that
\[
  |f (\phi)|_2\leq  |\Delta \phi|_2+ (|\phi|_6^3+|\phi|_2)
\]
holds. Taking into account the inequality $(a+b)^2\leq 2 a^2+2b^2$,  there exists a constant $d_4>0$ such that 
\[
  I_4\leq  d_4\left(|\Delta \phi|_2^2+|\phi|_6^6+|\phi|_2^2\right)|\psi|_2^2.
\]
For $I_5$ we can use H\"older inequality to get
\begin{equation} \label{eq.I5}
   I_5 \leq M_2\left(\mathcal{B} (\phi)+b \right)|\qspace|^{1/2}|f' (\phi)\psi|_2.
\end{equation}
Since 
\[ 
  \mathcal{B} (\phi)=\frac{1}{2}|\nabla\phi|_2^2+\frac{1}{4}|\phi^2-1|_2^2
\]
using Young inequality and H\"older inequality we get
\begin{equation*} \label{eq.I5a}
  \mathcal{B} (\phi)    \leq    \frac{1}{2}|\nabla\phi|_2^2+\frac{1}{2 }\left(|\phi^2|_2^2+1\right)
       =\frac{1}{2}|\nabla\phi|_2^2+\frac{1}{2 }\left(|\phi|_4^4+1\right)
\end{equation*}
For the last term on the right-hand side we can argue as for $I_1$ to get  
\begin{eqnarray} \label{eq.I5b}
   |\left(f' (\phi)\psi\right)|_2    &\leq&       |\Delta\psi|_2 + |(3\phi^2-1)\psi|_2 \notag\\
   &\leq &    |\Delta\psi|_2 + |3\phi^2-1|_2|\psi|_\infty \notag\\
   &\leq&    |\Delta\psi|_2 + \left(3|\phi|_4^2+1\right)|\psi|_\infty 
\end{eqnarray}
Taking into account \eqref{eq.I5}, \eqref{eq.I5a} and \eqref{eq.I5b}, the term $I_5$ is bounded by
\[
  I_5\leq M_2\left( \frac{1}{2}|\nabla\phi|_2^2+\frac{1}{2 }\left(|\phi|_4^4+1\right)+b \right)
               \left(   |\Delta\psi|_2 + \left(3|\phi|_4^2+1\right)|\psi|_\infty  \right)
\]
Elementary calculus and inequality $ab\leq a^2/2+b^2/2$ show that for some constant $c>0$, 
\[
  \left( \frac{1}{2}|\nabla\phi|_2^2+\frac{1}{2}\left(|\phi|_4^4+1\right)\right) \left(3|\phi|_4^2+1\right)
  \leq c\left(|\nabla\phi|_2^4+|\phi|_4^4+1 \right).
\]
Then there exists a constant $d_5$ such that 
\[
  I_5\leq d_5\left( |\nabla\phi|_2^4+|\phi|_4^4+1   \right)\left(|\psi|_2 +|\psi|_\infty\right)
\]
Summing up the bounds for $I_1, I_2,I_3,I_4,I_5$ and taking into account the Poincaré inequality $|\psi|_\infty\leq C_p|\nabla\psi|_2$ the result follows.
  \end{proof}


\subsection{Trace estimates}

We recall that we have made the following assumption on the operator $\Ca$ and $\Cb$.
\[ 
  \Tr[\Ca^*\Ca]<\infty,\qquad \Tr[\Cb^*\Delta^2 \Cb]<\infty.
\]
Since for $X,Y$ two separable Hilbert spaces and $A:X\to Y$ a linear operator, we have
 \[
  \Tr[A^*A] = \Tr[AA^*].
 \]
thus $\Tr[\Cb^*\Delta^2\Cb]<\infty \Longleftrightarrow \Tr[\Delta \Cb\Cb^*\Delta]<\infty$.

\begin{Rem} \label{rem.C}
 If $\Tr[\Cb^*\Delta^2\Cb]<\infty $, then $\Cb^*\Delta: \sH^2(\qspace)\cap \sH^1_0(\qspace)\to \sK$ {\em (in the definition of the noise, we usually have $\Cb: \sK\to \sH$, with $\sK\neq \sH$ in general)} 
 is closable and can be  extended to a bounded linear operator $\overline{\Cb^*\Delta}:\sH\to \sK$.
\end{Rem}
\begin{proof}
  Let $u\in \sH^2(\qspace)\cap \sH^1_0(\qspace)$ and let $(e_k)_k$ be a orthonormal basis of $\sK$. Then
  \[
    \|\Cb^*\Delta u\|_\sK^2=\sum_{k}\langle  \Cb^*\Delta u, e_k\rangle^2= \sum_{k}\langle  u,\Delta \Cb e_k\rangle^2\leq \|u\|_{\sH}^2 \sum_{k}| \Delta \Cb e_k|_K^2
    = \|u\|_\sH^2 \Tr[\Cb^*\Delta^2 \Cb ].
  \]
( notice that we have used the same notation $\langle \cdot,\cdot \rangle$ for the scalar product in $K$ and in $H$ ).
Then by the closed graph theorem we obtain the result. 
\end{proof}

\begin{Prop} \label{Prop.trace}
If $\Tr[\Cb^{*}\Delta^{2} \Cb]<\infty$ then $\Tr[\Cb^{*}\Cb] <+\infty$. Moreover the sequences
\[
( |\Cb e_{i} |_{\infty})_{i\in\N}, \quad ( \|\Cb e_{i} \|_{W^{2,2}(\qspace)})_{i\in\N}, \quad ( |\nabla C e_{i} |_{3})_{i\in\N}
\]
are in $\ell^{2}(\N)$ and there exists a constant $c>0$ such that
\begin{equation*} 
  \sum_{i\in \N} \left( |\Cb e_{i} |_{\infty}^2  +\|\Cb e_{i}\|_{W^{2,2}(\qspace)}^2+ |\nabla \Cb e_{i} |_{3}^2\right)   \leq    c\Tr[\Cb^{*}\Delta^{2} \Cb].
\end{equation*}
\end{Prop}

\begin{Proof}
Since $e_{i}$ is a component of an orthonormal basis then $\la \Cb^{*}\Delta^{2} \Cb e_{i},e_{i}\ra \leq \Tr[\Cb^{*}\Delta^{2}\Cb] <+\infty$ i.e. $\Delta \Cb e_i\in \sL^{2}(\qspace)$. This implies that
\begin{eqnarray*}
 \Tr[\Cb^*\Cb]= |\Cb e_0 |_2^2+\sum_{i=1}^\infty  |\Cb e_i  |_2^2
 &\leq& |\Cb e_0 |_2^2+\|(-\Delta)^{-1}\|_{\mathcal{L}(H)}^2\sum_{i=1}^\infty  |\Delta \Cb e_i |_2^2\\
 & \leq& \max\{1,\|(-\Delta)^{-1}\|_{\mathcal{L}(H)}^2\} \Tr[\Cb^{*}\Delta^{2} \Cb],
\end{eqnarray*}
and in particular $\Cb e_{i} \in W^{2,2}(\qspace)$. Denote $M =\max\{1,\|(-\Delta)^{-1}\|_{\mathcal{L}(H)}^2\}$. 
Moreover by Sobolev embedding $W^{2,2}(\qspace) \subset \mathcal{C}^{0,\gamma}$ for all $\gamma<1/2$ then there exists $M'>0$ such that
\[
\sum_{i=1}^\infty  |\Cb e_i |_\infty^2 \leq M'\sum_{i=1}^\infty \|\Cb e_i\|_{W^{2,2}(\qspace)}^2 \leq  4MM'\sum_{i=1}^\infty  |\Delta \Cb e_i |_2^2 \leq 4MM'\Tr[\Cb^{*}\Delta^{2} \Cb].
\]
Finally by Sobolev embedding $W^{1,\frac{3}{2}}(\qspace) \subset \sL^{3}(\qspace)$ and H\"older's inequality then there exists $M''>0$ such that
\begin{eqnarray*}
\sum_{i=1}^\infty \|\nabla \Cb e_i\|_3^2
&\leq& M''\sum_{i=1}^\infty \|\Cb e_i\|_{W^{1,3/2}(\qspace)}^2
\leq  M''\sum_{i=1}^\infty \left(|\Cb e_i |_{3/2}^{2} +  |\nabla Ce_i |_{3/2}^{2}\right)\\
&\leq&  M''(|\qspace|)^{1/3}\sum_{i=1}^\infty \left( |\Cb e_i |_{2}^{2} +  |\nabla \Cb e_i |_{2}^{2}\right)
\leq  M''(|\qspace|)^{1/3}\sum_{i=1}^\infty \|\Cb e_i\|_{W^{2,2}(\qspace)}^{2} < +\infty.
\end{eqnarray*}
\end{Proof}

\begin{Prop} \label{prop.trace}
Under hypothesis of proposition 2.1, there exists a constant $c>0$, depending on the operator $\Cb$, such that for any $\phi+1\in \mathcal{C}_0^\infty(\qspace)$ it holds
\[
\Tr\left[\Cb\Cb^*\dfrac{\delta^2 E (\phi)}{\delta \phi^2}\right]
\leq  c\left(|\Delta\phi|_2^2+|\nabla\phi|_2^4+|\phi|_4^2+1\right) \Tr[\Cb^{*}\Delta^{2} \Cb]
\]
\end{Prop}
\begin{Proof}
By Lemma \ref{Le.trace}, for any eigenvector $e_i$ we have
\[
\langle \Cb\dfrac{\delta^2 E (\phi)}{\delta \phi^2},\Cb e_i\rangle
\leq  c\left(|\Delta\phi|_2^2+|\nabla\phi|_2^4+|\phi|_4^2+1\right)
\left(|\Cb e_i|_2^2+|\nabla \Cb e_i|_{2}^2+|\Delta \Cb e_i|_{2}^2\right)
\]
By taking the sum over $i$ and using Proposition \ref{Prop.trace} we get the result.  
\end{Proof}

%

\section{Existence of a solution - preliminaries}

\subsection{Approximated equation and a priori estimates}

Before proceeding to the proof, we need an approximation of equation \eqref{eq.aNSPF}.

Let us choose $\{e_{j}\}_{j\in\N^{*}} \in \sH$ to be
the eigenfunctions of the Stokes operator with homogeneous boundary conditions, such that $\{e_{j}\}_{j\in\N^{*}}$ forms an orthonormal basis for
$\sH$. 
Let also $\{\eta_j\}_{j\in\N^{*}} \in \sL^2(\qspace)$ be the orthonormal basis in $\sL^2(\qspace)$  consisting of the eigenfunctions of the Laplacian $\Delta$ with homogeneous Dirichlet boundary conditions.

Next, set $S_n=\mathrm{span }\{e_1,\ldots,e_n\}$, $N_n=\mathrm{span }\{\eta_1,\ldots,\eta_n\}$.
Finally, we denote by $P_n : \sH\to \sH$ the orthogonal projection of $\sH$ to $S_n$, and by $\pi_n:\sL^2(\qspace)\to \sL^2(\qspace)$ the orthogonal projection of $\sL^2(\qspace)$ into $N_n$.

We consider the equations
\begin{equation}  \label{eq.approx}
\begin{cases}
\dd(w_n+\alpha^2 A w_n)=P_{n}\Ca\dd W(t) + \\
\left(-\nu A(w_n+\alpha^2 A w_n)- P_n\widetilde{B}_{}(w_n,w_n+\alpha^2 A w_n) +P_{n}\left(\pi_n\left(\dfrac{\delta E_{}(\phi_{n})}{\delta \phi}\right)\nabla \phi_n\right) \right)\ddt, & \text{in } [0,T]\times \qspace,\\
\dd\phi_n=\left(-\pi_n\left(w_n\cdot \nabla \phi_n\right)  -\gamma \pi_{n} \left(\dfrac{\delta E_{}(\phi_{n})}{\delta \phi_n}\right)\right)\ddt+\pi_{n}\Cb\dd W'(t), &\text{in } [0,T]\times \qspace, \\
 w_n(0)=P_nw_0, &  \text{in } \qspace, \\
 \phi_n(0)=\pi_n(\phi_0+1)-1,&  \text{in } \qspace. 
 \end{cases}
\end{equation}



Equation \eqref{eq.approx} is a system of ordinary stochastic differential equations with polynomial nonlinear coefficients. 
Therefore, there exists a unique local strong  solution $(w_n,\phi_n)$ defined up to a blow up random time $\tau(\omega)$.
In order to show global existence and uniqueness of a solution for the approximated equations, we shall show a priori estimates.

By applying formally (exact proof is in the next section) the Itô formula we find
\begin{eqnarray}
  d\left(|w_n|_2^2+\alpha^2 |\nabla w_n|_2^2\right) & = &  \Bigg( \nu(|\nabla w_n|_2^2 +\alpha^2|A  w_n|_2^2)+\la P_n\widetilde{B}_{}(w_n,w_n+\alpha^2 A w_n), w_n\ra \notag \\
                                                    &   &  +\la P_{n}\left(\pi_n\left(\dfrac{\delta E(\phi_{n})}{\delta \phi}\right)\nabla \phi_n\right),w_n\ra +  \frac12 \Tr[(P_n\Ca)^*(I+\alpha^2A)^{-1} (P_n\Ca)]\Bigg)\dd t \notag\\
                                                    &   &  + \la w_n, (P_n\Ca)dW(t) \ra \label{eq.itow_n}
\end{eqnarray}
and 
\begin{eqnarray}
  d E(\phi_n)   &=&   \Bigg( -\la  \pi_n(\nabla \phi_n \cdot w_n) ,   \dfrac{\delta E(\phi_{n})}{\delta \phi}  \ra -\gamma\left|\dfrac{\delta E(\phi_{n})}{\delta \phi}\right|_2^2 
                                          + \frac12\Tr\left[(\pi_n \Cb)^*(\pi_n \Cb)\dfrac{\delta^2 E_{}(\phi_{n})}{\delta \phi^2}\right]\Bigg) \dd t \notag   \\
                && + \la\dfrac{\delta E_{}(\phi_{n})}{\delta \phi},(\pi_n\Cb)\dd W'(t)\ra \label{eq.itophi_n}
\end{eqnarray}
Notice that by Proposition \ref{prop.1.1} and by the fact that  $w_n\in S_n$ we have 
\[ 
    \la P_n\widetilde{B}_{}(w_n,w_n+\alpha^2 A w_n), w_n\ra = \la  \widetilde{B}_{}(w_n,w_n+\alpha^2 A w_n), w_n\ra =0
\]
 and
  \[
     \la P_{n}\left(\pi_n\left(\dfrac{\delta E(\phi_{n})}{\delta \phi}\right)\nabla \phi_n\right),w_n\ra = \la  \pi_n\left(\dfrac{\delta E(\phi_{n})}{\delta \phi}\right),\nabla \phi_n \cdot w_n\ra  
         = \la \dfrac{\delta E(\phi_{n})}{\delta \phi}, \pi_n\left(\nabla \phi_n \cdot w_n\right)\ra
  \]
Then, by summing up \eqref{eq.itow_n} and \eqref{eq.itophi_n} we find 
\begin{equation}
 \begin{split}\label{eq.ito}
  & d\left(|w_n|_2^2+\alpha^2 |\nabla w_n|_2^2 + E(\phi_n)\right) =    \Bigg( \nu(|\nabla w_n|_2^2 +\alpha^2|A  w_n|_2^2)-\gamma\left|\dfrac{\delta E(\phi_{n})}{\delta \phi}\right|_2^2  \\
   & \quad + \frac12 \Tr[(P_n \Ca)^*(I+\alpha^2A)^{-1} (P_n \Ca)]  +  \frac12\Tr\left[(\pi_n \Cb)^*(\pi_n \Cb)\dfrac{\delta^2 E_{}(\phi_{n})}{\delta \phi^2}\right] \Bigg)\dd t   \\   
   & \quad   + \la w_n, (P_n \Ca)dW(t) \ra + \la\dfrac{\delta E_{}(\phi_{n})}{\delta \phi},(\pi_n \Cb)\dd W'(t)\ra    
 \end{split}
\end{equation}

\subsection{Existence and uniqueness for the approximated equation}

\begin{Th} \label{thm.moments}
 Let $(w_0,\phi_0)\in D(A)\times \sL^2(\qspace)$ and assume that Hypothesis \ref{hyp.trace} holds.
 Then, for any $n\in \N$, $T>0$ there exists a solution $(w_n,\phi_n)\in \sL^2([0,T];D(A))\times  \sL^2([0,T];\sL^2(\qspace))$ of problem \eqref{eq.approx}.
 Moreover,  for any $T>0$, $k\in \N^*$ there exists a constant $c=c(k,T,\phi_0,w_0)>0$ such that for any $n\in \N^*$
 \begin{eqnarray*}
  &&\sup_{0\leq t\leq T}\E\left[\left( |w_n(t)|_2^2 +\alpha^2 |\nabla w_n(t)|_2^2+E(\phi_n(t))\right)^k \right] \leq c\\
  &&\E\left[\int_0^{T} \left( |w_n|_2^2 +\alpha^2 |\nabla w_n|_2^2+E(\phi_n)\right)^{k-1} \left(\nu(|\nabla w_n|_2^2       +\alpha^2|A  w_n|_2^2 )
     +\gamma\left|\dfrac{\delta E(\phi_{n})}{\delta \phi}\right|_2^2\right) \dd s\right]\leq c
 \end{eqnarray*}
\end{Th}
%
%
\begin{proof}
Set 
\begin{equation}  \label{eq.defF}
   \mathcal{F}(t)=\mathcal{F}(t,w_n,\phi_n)=  |w_n|_2^2 +\alpha^2 |\nabla w_n|_2^2+E(\phi_n) 
\end{equation}
For any $N>0$, $n\in \N^*$ we consider the stopping time
\[
   \tau_N^n=\inf\{t:\mathcal{F}(t,w_n,\phi_n)>N\}.
\]
As pointed out previously, \eqref{eq.approx} is a system of ordinary differential equations with polynomial nonlinearities.
Then, there exists a local solution $(w_n,\phi_n)$ up to a blow up time $\tau(w)$. 
Since the functions $w_n(t\wedge \tau_N^n), \phi_n(t\wedge\tau_N^n )$ are bounded by $N$, we can apply the  Itô formula in \eqref{eq.ito} to obtain 
\begin{eqnarray*}
  \mathcal{F}^k(t \wedge \tau_N^n) &+& 2k\int_0^{t \wedge \tau_N^n}\mathcal{F}^{k-1} \times  \left(\nu(|\nabla w_n|_2^2       +\alpha^2|A  w_n|_2^2 )
     +\gamma\left|\dfrac{\delta E(\phi_{n})}{\delta \phi}\right|_2^2\right) \dd s \\
     &=&  2k\int_0^{t \wedge \tau_N^n}\mathcal{F}^{k-1} \times\left( \frac12 \Tr[(P_n \Ca)^*(I+\alpha^2A)^{-1}(P_n \Ca)]
              + \frac12\Tr\left[(\pi_n \Cb)^*(\pi_n \Cb)\dfrac{\delta^2 E_{}(\phi_{n})}{\delta \phi^2}\right]   \right)\dd s\\
   &&  + 2k\int_0^{t \wedge \tau_N^n}\mathcal{F}^{k-1}\la w_n,(P_n \Ca)\dd W(s)\ra
       +2k\int_0^{t \wedge \tau_N^n} \mathcal{F}^{k-1}\la\dfrac{\delta E_{}(\phi_{n})}{\delta \phi},(\pi_n \Cb)\dd W'(s)\ra    \\
    && + k(k-1)  
        \int_0^{t \wedge \tau_N^n}\mathcal{F}^{k-2}  \left( |(P_n \Ca)^*w_n|_2^2 +\left|(\pi_n \Cb^*)\dfrac{\delta E(\phi_{n})}{\delta \phi} \right|_2^2 \right)\dd s\\
     &=& I_1+I_2+M_t
\end{eqnarray*}
where $M_t$ is the martingale term.
Let us estimate $I_1$.  
By  Proposition \ref{prop.trace} there exists $c_1>0$ such that
\[
   \frac12\Tr\left[(\pi_n \Cb)^*(\pi_n \Cb)\dfrac{\delta^2 E_{}(\phi_{n})}{\delta \phi^2}\right]
    \leq  c_1\left(|\Delta\phi_n|_2^2+|\nabla\phi_n|_2^4+|\phi_n|_4^2+1\right) \Tr[\Cb^{*}\Delta^{2} \Cb].  
\]
By \eqref{eq.ineq.E2} and elementary inequalities there exists a positive constant $c_2$ such that 
\[
   \left(|\Delta\phi_n|_2^2+|\nabla\phi_n|_2^4+|\phi_n|_4^2+1\right)\leq c_2(1+E(\phi_n))\leq c_2(1+\mathcal{F}).
\]
   Taking into account that $\Tr[\Cb^{*}\Delta^{2} \Cb]$ and $\Tr[\Ca^*(I+\alpha^2A)^{-1} \Ca]  $ are bounded, there exists  $c_3>0$ that
\[
 I_1\leq c_3\int_0^{t \wedge \tau_N^n}(\mathcal{F}^{k-1}+\mathcal{F}^k)\dds
\]
Let us estimate $I_2$. 
By \eqref{eq.CEestimate} there exists $c_4>0$ such that 
\[
    \left|(\pi_n \Cb)^*\dfrac{\delta E(\phi_{n})}{\delta \phi} \right|_2^2\leq c_4\left(1+|\phi_n|_4^8+|\phi_n|_6^3+|\nabla\phi_n|_2^4+|\Delta \phi_n|_2^2\right)^2. 
\]
Using \eqref{eq.ineq.E2}, the quantity on the right hand side is bounded by $c(1+E(\phi_n)^2)$, for a suitable $c>0$ independent by $\phi$. 
By elementary inequalities and the fact that the operators $P_n \Ca^*$ are uniformly bounded with respect to $n$, 
we deduce that for there exists $c_5>0$, independent by $n$, $\phi_n$, $w_n$ such that
\[
 I_2 \leq c_5\int_0^{t\wedge \tau_N^n}\left(\mathcal{F}^{k-2}+\mathcal{F}^{k}\right)\dd s.
\]
Finally,
\begin{equation} \label{eq.boundsF}
 \begin{split}
    \mathcal{F}^k(t \wedge \tau_N^n)  \leq&  c_3\int_0^{t \wedge \tau_N^n}(\mathcal{F}^{k-1}+\mathcal{F}^k)\dds
                                +c_5\int_0^{t\wedge \tau_N^n}\left(\mathcal{F}^{k-2}+\mathcal{F}^{k}\right)\dd s\\
          &                      +2k\int_0^{t \wedge \tau_N^n}\mathcal{F}^{k-1}\la w_n,(P_n \Ca)\dd W(s)\ra
       +2k\int_0^{t \wedge \tau_N^n} \mathcal{F}^{k-1}\la\dfrac{\delta E_{}(\phi_{n})}{\delta \phi},(\pi_n \Cb)\dd W'(s)\ra
 \end{split}
\end{equation}
Befor  taking expectation, we need to verify that the martingale terms are integrable.
Notice that since the operator $\Ca$ is bounded there exists $c_6>0$ such that
\[
 \mathcal{F}^{k-1}|(P_n \Ca)^* w_n|_2\leq c_6(1+\mathcal{F}^{k}).
\]
Then, since $\mathcal{F}^k(t\wedge \tau_n)\leq N^k$, we can take expectation  to obtain
\[ 
    2k\E\int_0^{t \wedge \tau_N^n}\mathcal{F}^{k-1}\la w_n,(P_n \Ca)\dd W(s)\ra=0.
\]
Similarly, for the second term we can use estimate \eqref{eq.CEestimate} and obtain, for some $c_6>0$
\[
    \mathcal{F}^{k-1}\left|(\pi_n \Cb)^*\dfrac{\delta E_{}(\phi_{n})}{\delta \phi}\right|_2
     \leq c_6\mathcal{F}^{k-1}\left(1+|\phi_n|_4^8+|\phi_n|_6^3+|\nabla\phi_n|_2^4+|\Delta \phi_n|_2^2\right)
\]
As we pointed out previously, by \eqref{eq.ineq.E2}  there exists $c>0$ such that
\[
  \left(1+|\phi|_4^8+|\phi_n|_6^3+|\nabla\phi_n|_2^4+|\Delta \phi_n|_2^2\right)\leq c(1+E(\phi_n))\leq c(1+\mathcal{F}).
\]
Then, there exists $c_7>0$ such that
\[
   \mathcal{F}^{k-1}\left|(\pi_n \Cb)^*\dfrac{\delta E_{}(\phi_{n})}{\delta \phi}\right|_2\leq c_7\mathcal{F}^{k-1}(1+\mathcal{F}).
\]
This implies that we can take expectation to obtain
\[ 
  2k\E \int_0^{t \wedge \tau_N^n} \mathcal{F}^{k-1}\la\dfrac{\delta E_{}(\phi_{n})}{\delta \phi},(\pi_n \Cb)\dd W'(s)\ra   =0.
\]
Finally, by taking expectation in \eqref{eq.boundsF} we get 
\begin{multline} \label{eq.boundsF2}
  \E[\mathcal{F}^k(t \wedge \tau_N^n)] +2k\E\left[\int_0^{t \wedge \tau_N^n}\mathcal{F}^{k-1} \times  \left(\nu(|\nabla w_n|_2^2       +\alpha^2|A  w_n|_2^2 )
     +\gamma\left|\dfrac{\delta E(\phi_{n})}{\delta \phi}\right|_2^2\right) \dd s \right] \\
   \leq  c_3\E\left[\int_0^{t \wedge \tau_N^n}(\mathcal{F}^{k-1}+\mathcal{F}^k)\dds\right]
                                +c_5\E\left[\int_0^{t\wedge \tau_N^n}\left(\mathcal{F}^{k-2}+\mathcal{F}^{k}\right)\dd s\right].
\end{multline}
Clearly, there exists a constant $c>0$ such that $\mathcal{F}^{k-1}\leq c(1+\mathcal{F}^k)$ and $\mathcal{F}^{k-2}\leq c(1+\mathcal{F}^{k})$.
Then, there exists $c_7>0$, depending only on $k$,   $\phi_0$, $w_0$, such that the right-hand side of \eqref{eq.boundsF2} is bounded by
\[
  c_7\E\left[\int_0^{t \wedge \tau_N^n}(1+\mathcal{F}^k)\dds\right]\leq c_7\E\left[\int_0^{t }(1+\mathcal{F}^k(s\wedge \tau_N^n))\dds\right].
\]
Using Gronwall lemma, we find that there exists a constant $c_8>0$ depending on $k$, $T$,   $\phi_0$, $w_0$, such that
\[
  \sup_{t\in [0,T]}E[\mathcal{F}^k(t \wedge \tau_N^n)]+2k\int_0^{T \wedge \tau_N^n}\mathcal{F}^{k-1} \times  \left(\nu(|\nabla w_n|_2^2       +\alpha^2|A  w_n|_2^2 )
     +\gamma\left|\dfrac{\delta E(\phi_{n})}{\delta \phi}\right|_2^2\right) \dd s \leq c.
\]
Letting $N\to\infty$ we conclude the proof.
\end{proof}
\begin{Th}   \label{thm.montentsup}
 Let $(w_0,\phi_0)\in D(A)\times \sL^2(\qspace)$ and assume that Hypothesis \ref{hyp.trace} holds.
 Then for any $T>0$, $k\in \N$ there exists $c=c(k,T,w_0,\phi_0)>0$ such that
 \[
  \E\left[\sup_{t\in[0,T]}\left( |w_n|_2^2 +\alpha^2 |\nabla w_n|_2^2+E(\phi_n)\right)^k \right] \leq c.
 \]
\end{Th}
\begin{proof}
As done for the previous Theorem, let us set $\mathcal{F}$ as in \eqref{eq.defF}.
By Theorem \ref{thm.moments} the solution $(w_n,\phi_N)$ is global and all moments of $\mathcal{F}$ have finite expectation.
Then by  Itô formula \eqref{eq.ito} we get 
\begin{eqnarray*}
  \mathcal{F}^k(t) &=& 2k\int_0^{t}\mathcal{F}^{k-1} \times \left( \left(-\nu(|\nabla w_n|_2^2       +\alpha^2|A  w_n|_2^2 )
     -\gamma\left|\dfrac{\delta E(\phi_{n})}{\delta \phi}\right|_2^2\right.\right. \\ 
    && +  \left.\left.  \frac12\Tr[(P_n \Ca)^*(I+\alpha^2A)^{-1} (P_n \Ca)] 
      + \frac12\Tr\left[(\pi_n \Cb)^*(\pi_n \Cb)\dfrac{\delta^2 E_{}(\phi_{n})}{\delta \phi^2}\right]   \right)\dd s \right)\\
      && + k(k-1)  
        \int_0^{t}\mathcal{F}^{k-2}  \left( |(P_n \Ca)^*w_n|_2^2 +\left|(\pi_n \Cb^*)\dfrac{\delta E(\phi_{n})}{\delta \phi} \right|_2^2 \right)\dd s\\
   &&  + 2k\int_0^{t}\mathcal{F}^{k-1}\la w_n,(P_n \Ca)\dd W(s)\ra
       +2k\int_0^{t} \mathcal{F}^{k-1}\la\dfrac{\delta E_{}(\phi_{n})}{\delta \phi},(\pi_n \Cb)\dd W'(s)\ra    \\
     &=& I_1+I_2+M_t
\end{eqnarray*}
Where $I_1$, $I_2$ are the integrals containing $\mathcal{F}^{k-1}$ and  $\mathcal{F}^{k-1}$ respectively, and $M_t$ is the martingale term.
As we done for Theorem \ref{thm.moments}, $I_1$, $I_2$ are uniformly bounded in $t$ by
\[
  I_1+I_2\leq c\int_0^{T}(1+\mathcal{F}^k)\dd s,
\]
where $c>0$ is a suitable constant depending only by $k,T$.
For the martingale part, we can use Burkholder-Davis-Gundy inequality  to get  for some constant $c_1,c_2>0$
\begin{eqnarray*}
  \E\left(\sup_{t\in[0,T]}\left|\int_0^{t}\mathcal{F}^{k-1}\la w_n,(P_n \Ca)\dd W(s)\ra\right|\right) 
  &\leq& c_1\E\left(\int_0^{T} \mathcal{F}^{2(k-1)}\left|(P_n \Ca)^*w_n\right|_2^2\dd s \right)^\frac{1}{2} \\
  &\leq&  c_2\E\left(\int_0^{T} \mathcal{F}^{2k}\dd s \right)^\frac{1}{2} <\infty
\end{eqnarray*}
The last term is bounded thanks to Theorem \ref{thm.moments}.
Again, by Burkholder-Davis-Gundy inequality  there exists $c_3>0$ such that

\begin{eqnarray*}
   \E \left(\sup_{t\in[0,T]}\left|\int_0^{t} \mathcal{F}^{k-1}\la\dfrac{\delta E_{}(\phi_{n})}{\delta \phi},(\pi_n \Cb)\dd W'(s)\ra\right| \right) 
  \leq c_3\E\left(\int_0^{T} \mathcal{F}^{2(k-1)}\left|(\pi_n \Cb)^*\dfrac{\delta E_{}(\phi_{n})}{\delta \phi}\right|_2^2\dd s \right)^\frac{1}{2}
\end{eqnarray*}
By estimate \eqref{eq.CEestimate} and \eqref{eq.ineq.E2}, there exists $c_4>0$ such that the right-hand side is bounded by
\begin{equation*}
  c_4\E\left(\int_0^{T} \mathcal{F}^{2(k-1)}(1+\mathcal{F}^2)\dd s \right)^\frac{1}{2}.
\end{equation*}
Then, by \ref{thm.moments} this integral is finite. 
This complete the proof.
\end{proof}

\subsection{Compactness argument - convergence to a solution }
 Let $X$ be a Banach space with norm $\|\cdot\|_X$.
 For $p\geq1$, $\theta\in]0,1[$ we denote by $W^{\alpha,p}([0,T];X)$ classical Sobolev space of all functions $f\in \sL^p([0,T];X)$ such that
 \[
   \int_0^T\int_0^T\frac{|f(t)-f(s)|_2^p}{|t-s|^{1+\theta p}}\dd s\dd t <\infty,
 \]
 endowed with the norm
 \[
   \|f\|_{W^{\theta,p}([0,T];X)}=\left(\|f\|_{L^p([0,T];X)}^p+\int_0^T\int_0^T\frac{\|f(t)-f(s)\|_X^p}{|t-s|^{1+\theta p}}\dd s\dd t  \right)^{\frac1p}.
 \]
The proof of the following lemma is left to the reader
\begin{Le} \label{Le.sobolev}
Let $X$ a  Banach space.
For any $\theta\in]0,1/2[$ $p\geq1$ there exists $c=c(\theta,p)$ such that for any $f\in \sL^2([0,T];X)$ it holds 
 \[
   \left\|\int_0^\cdot f(\tau)\dd \tau \right\|_{W^{\theta,p}([0,T];X)}\leq c(\theta,p) \|f\|_{L^2([0,T];X)}
 \]
\end{Le}

\begin{Prop}   \label{prop.holderw}
  For any $T>0$, $\theta \in ]0,1/2[$, $p\geq1$ there exists $c=c(T,\theta,p)>0$ such that for any $n\in \N$ 
\[
   \E\left[ \|w_n+\alpha^2Aw_n\|_{W^{\theta,p}([0,T];D(A)')}^2\right]\leq c.
\]
\end{Prop}


\begin{proof}
For any $n\in \N $, $\xi\in D(A)$ we have
 \begin{eqnarray*}
  \la w_n(t)+\alpha^2Aw_n(t),\xi \ra_{(D(A)',D(A))}  &=&   
  -\nu\int_0^t    \la w_n(\tau)+\alpha^2 A w_n(\tau),A\xi\ra\dd \tau\\
   && - \int_0^t\la P_n\widetilde{B}_{}(w_n  ,w_n +\alpha^2 A w_n)(\tau),\xi\ra \dd\tau \\
   && +\int_0^t \la P_{n}\left(\dfrac{\delta E_{{}}(\phi_{n})}{\delta \phi}\nabla \phi_n\right),\xi \ra  \dd t\\
   && +\la (P_n \Ca) W(t),\xi\ra \\
   &=&  J_1(t)+J_2(t)+J_3(t)+J_4(t).
 \end{eqnarray*}
We proceed as for Proposition \ref{prop.holderw} by estimating each term.
For $J_1$ we have, using Lemma \ref{Le.sobolev} and Theorem \ref{thm.moments} (with $k=1$), that there exists $c_1>0$ such that
\[
 \E\left[ \| J_1(\cdot)\|_{W^{\theta,p}([0,T];\R)}^2\right]
   \leq c(\theta,p)\E\left[\int_0^T\left(|w_n(\tau)|_2+\alpha^2|Aw_n(\tau)|_2^2\right)\dd \tau\right]|A\xi|_2^2
     \leq c_1|\xi|_{D(A)}^2
\]
In order to estimate $J_2$, observe that by {\em (iv)} of Proposition \ref{prop.1.1} and Young inequality, we have 
\begin{eqnarray*}
  \la  P_n\widetilde{B} (w_n ,w_n +\alpha^2 A w_n),\xi \ra_{(D(A)',D(A))}\dd   &=&    \la \widetilde{B} (w_n ,w_n +\alpha^2 A w_n),P_n\xi \ra_{(D(A)',D(A))} \\
   &\leq&    c  |w_n|_V\left(|w_n|_2 +\alpha^2 |A w_n|_2\right) |\xi|_{D(A)}  
\end{eqnarray*}
By Lemma \ref{Le.sobolev} and the bound given by Theorem \ref{thm.moments}, we deduce that there exists $c_2>0$ such that 
\begin{eqnarray*}
   \E\left[ \| J_2(\cdot)\|_{W^{\theta,p}([0,T];\R)}^2\right]  \leq 
          c \E\left[\int_0^T|w_n|_V^2\left(|w_n|_2 +\alpha^2 |A w_n|_2\right)^2\dd\tau \right] |\xi|_{D(A)}^2 
          \leq c_2|\xi|_{D(A)}^2.
\end{eqnarray*}
In order to estimate  $J_3$, let us obverse that  we have, by Hölder and Sobolev inequalities  ( which works both in dimensions $2$ and $3$ )
\begin{eqnarray*}
 \left| \la P_{n}\left(\dfrac{\delta E_{}(\phi_{n})}{\delta \phi}\nabla \phi_n\right),\xi \ra_{(D(A)',D(A))}\right|
     &\leq& \left|P_{n}\left(\dfrac{\delta E_{}(\phi_{n})}{\delta \phi}\right)\right|_2  \left|\nabla \phi_n\right|_3   |\xi|_6
\\
  &\leq& \left| \dfrac{\delta E_{}(\phi_{n})}{\delta \phi} \right|_2  \left\|\nabla \phi_n\right\|_{\sH^1(\qspace)}   \|\xi\|_{\sH^1(\qspace)}\\
  &\leq&  c\left| \dfrac{\delta E_{}(\phi_{n})}{\delta \phi} \right|_2  \left\|\phi_n\right\|_{\sH^2}   |\xi|_{D(A)}\\
  &\leq&  c\left| \dfrac{\delta E_{}(\phi_{n})}{\delta \phi} \right|_2  (1+E(\phi_n))   |\xi|_{D(A)}.
\end{eqnarray*}
In the last inequality we used \eqref{eq.ineq.E2}.
Then, by Lemma \ref{Le.sobolev} and the estimates in Theorem \ref{thm.moments} (with $k=3$), we deduce that there exists $c_3>0$ such that
\begin{eqnarray*}
    \E\left[  \| J_3(\cdot)\|_{W^{\theta,p}([0,T];\R)}^2  \right] 
      & \leq & \E\left[ \int_0^T\left| \la P_{n}\left(\dfrac{\delta E_{}(\phi_{n})}{\delta \phi}\nabla \phi_n\right),\xi \ra_{(D(A)',D(A))}\right|^2\dd\tau\right] \\
         & \leq & c \E\left[ \int_0^T\left| \dfrac{\delta E_{}(\phi_{n})}{\delta \phi} \right|_2^2  (1+E(\phi_n))^2 \dd \tau\right]|\xi|_{D(A)}^2\\
         & \leq & c_3|\xi|_{D(A)}^2
\end{eqnarray*}   
The term $J_4$ is treated as done in \eqref{eq.sobolev.C}. Then, provided $\theta<1/2$, there exists $c_4>0$ such that
\[
  \E\left[\int_0^T\int_0^T\frac{|(P_n \Ca)(W(t)-W(s)|_2^p}{|t-s|^{1+\theta p}}\dd s\dd t\right]
    \leq c(\Tr[\Ca^*\Ca])^{\frac{p}{2}}\int_0^T\int_0^T \frac{|t-s|^p }{|t-s|^{1+\theta p}}\dd s\dd t  \leq c_3.
\]
Finally, the results follows by taking into account the estimates obtained for $J_1,J_2,J_3,J_4$.
\end{proof}

\begin{Prop} \label{prop.holderphi}
 For any $T>0$, $\theta\in ]0,1/2[$, $p\geq 1$ there exists $c=c(T,\theta,p)>0$ such that for any $n\in \N$ 
 \[
 \E\left[ \|\phi_n\|_{W^{\theta,p}([0,T];\sL^2(\qspace))}^2\right]\leq c.
 \]
\end{Prop}
\begin{proof}
For any $n$ we have
\[
      \phi_n(t)=\int_0^t \pi_n\left(w_n\nabla\phi_n\right)d\tau -\gamma \int_0^t\pi_n\left(\frac{\delta E}{\delta\phi}(\phi_n(\tau))\right)d\tau +(\pi_n \Cb)W'(t) 
        = K_1(t)+K_2(t)+K_3(t). 
\]
We proceed by estimating each term. 
For $K_1$ we have, using elementary inequalities
\[
   \int_0^T |\pi_n\left(w_n\nabla\phi_n\right)|_2^2\dd\tau\leq \sup_{0\leq t\leq T}|w_n|_\infty \int_0^T |\nabla\phi_n|_2^2\dd\tau\leq T \sup_{0\leq t\leq T}|w_n|_\infty^2+\int_0^T |\nabla\phi_n|_2^4\dd\tau.
\]
Then by Lemma \ref{Le.sobolev} and Theorem \ref{thm.montentsup} we deduce that there exists $c_1>0$, independent by $n$ such that
\begin{eqnarray*}
   \E\left[ \|K_1(\cdot)\|_{W^{\theta,p}([0,T];\sL^2(\qspace))}^2\right] &\leq&  c(\theta,2)\int_0^T |\pi_n\left(w_n\nabla\phi_n\right)|_2^2\dd\tau \\
    &\leq& T  c(\theta,2)\E\left[\sup_{0\leq t\leq T}|w_n|_\infty^2\right]+c(\theta,2)\E\left[\int_0^T |\nabla\phi_n|_2^4\dd\tau\right] \\
    &\leq& T  c(\theta,2)\E\left[\sup_{0\leq t\leq T}|w_n|_\infty^2\right]+c\E\left[\int_0^T \left(1+E(\phi_n(\tau))\right)\dd\tau\right] \leq  c_1.
\end{eqnarray*}
In the last inequality we used \eqref{eq.ineq.E2}.

For $K_2$ we have, by Lemma \ref{Le.sobolev} and Theorem \ref{thm.moments}, that for some $c_2>0$, independent by $n$, it holds  
\begin{eqnarray*}
  \E\left[ \|K_2(\cdot)\|_{W^{\theta,p}([0,T];\sL^2(\qspace))}^2\right]\leq  \E\left[ \left\|\frac{\delta E}{\delta\phi}(\phi_n(\tau))   \right\|_{L^2([0,T];\sL^2(\qspace))}^2\right]<c_2
\end{eqnarray*}

For the last term we have, by the gaussianity of $\Cb(W'(t)-W'(s))$ that there exists $c=c(p)$ such that 
$\E[ |(\pi_n \Cb)(W'(t)-W'(s)|_2^p] \leq c (\Tr[\Cb^*\Cb])^{\frac{p}{2}} |t-s|^{\frac{p}{2}}$.
Then, 
\begin{eqnarray} \label{eq.sobolev.C}
  \E\left[\int_0^T\int_0^T\frac{|(\pi_n \Cb)(W'(t)-W'(s)|_2^p}{|t-s|^{1+\theta p}}\dd s\dd t\right]
    &\leq& c(\Tr[\Cb^*\Cb])^{\frac{p}{2}}\int_0^T\int_0^T \frac{|t-s|^p }{|t-s|^{1+\theta p}}\dd s\dd t  \leq c_3
\end{eqnarray}
provided $\theta<1/2$.
Taking into account the estimates on $K_1$, $K_2$, $K_3$ we obtain the result.
\end{proof}

In which follows, we denote by  $L_w^2([0,T];D(A)')$ the space $L^2([0,T],D(A)')$ endowed with the weak $L^2$ topology.
\begin{Le}[Tightness] \label{le.tightness}
 For $(w_0,\phi_0)\in D(A)\times \sL^2(\qspace)$ with $\phi_0=-1$ on $\partial \qspace$, $T>0$, $n\in \N$, let $(w_n,\phi_n)$ the solution of \eqref{eq.approx} in $[0,T]$.
 Then, for any $p>2$, $\rho>0$, the laws of $w_n, n\in \N$ are tight in 
 \[
     \mathcal{C}([0,T];D(A^{-\rho}))\cap \sL^p([0,T];V)\cap \sL_w^2([0,T];D(A)')
 \]
Moreover, for any $\sigma>0$,  the laws of $\phi_n, n\in \N$ are tight in
\[
  \mathcal{C}([0,T];\sH^{-\sigma}(\qspace))\cap \sL^p([0,T];\sH^2(\qspace))\cap \sL_w^2([0,T];(\sH^4(\qspace))') .
\]
\end{Le}
\begin{proof}
The classical interpolation inequality
\[
   \|w\|_{\sH^{1+\rho}}\leq  \|w\|_{\sH^{1}}^{1-\rho} \|w\|_{\sH^{2}}^\rho, \qquad \rho\in [0,1]
\]
implies
\[
     \|w\|_{\sH^{1+\frac{2}{p}}}^{p}\leq  \|w\|_{\sH^{1}}^{p-2} \|w\|_{\sH^{2}}^2, \qquad p\in [2,\infty[.
\]
Then, by Theorem \ref{thm.moments} and Proposition \ref{prop.holderw} implies that $(w_n)_n$ is bounded in 
\begin{equation*} 
   \sL^p\left(\Omega;\sL^p([0,T];\sH^{1+\frac{2}{p}})\right)\cap \sL^2\left( \Omega; \sL^2([0,T];D(A))\right)\cap \sL^2\left( \Omega; W^{\theta, p}([0,T];\sH)\right)
\end{equation*}
for any $p\in]2,\infty[$ and $\theta<1/2$ such that $\theta p>1$.
Taking into account Theorem \cite[Theorem 2.1 and Theorem 2.2]{FlandoliGatarek}, for any $p\in]2,<\infty[ $ and $\theta<1/2$ such that $\theta p>1$ the embeddings 
\begin{eqnarray*}
  &&   W^{\theta, p}([0,T];\sH)  \hookrightarrow \mathcal{C}([0,T];D(A^{-\rho})),\qquad \rho>0\\
  &&  \sL^p([0,T];\sH^{1+\frac{2}{p}})     \cap   W^{\theta, p}([0,T];\sH)\hookrightarrow \sL^p([0,T];\sV)
\end{eqnarray*}
are compact.
Moreover, we have that $\sL^2([0,T];D(A))$ is compactly embedded in the complete metrizable space $\sL_w^2([0,T]; D(A)')$.
Then, the result follows by Prokhorov's theorem.

In order to show the tightness of the laws of $\phi_n$, 
notice that by \eqref{eq.ineq.E3} there exists $c>0$, independent by $n$, such that 
\[
  \E\left[\int_0^T|\Delta^2\phi_n|_2^2\ddt \right]\leq c \E\left[\int_0^T\left(\left|\dfrac{\delta E(\phi_{n})}{\delta \phi} \right|_2^2+1+(E(\phi))^2\right)\ddt \right].
\]
Taking into account Theorem \ref{thm.moments}, this implies that the sequence $(\phi_n)_n$ is uniformly bounded in $ \sL^2(\Omega;L^2([0,T];\sH^4(\qspace)))$
and then the laws of $\phi_n,n\in \N$ are tight in  the complete metrizable space  $L^2_w([0,T];(\sH^4(\qspace))')$.
By the interpolation formula $\|\phi\|_{\sH^{1+2\rho}} \leq \|\phi\|_{\sH^2}^{1-\rho}\|\phi\|_{\sH^4}^\rho$ we deduce that for some $c>0$
\[
   \|\phi\|_{\sH^{2+\frac{4}{p}}}^p\leq \|\phi\|_{\sH^2}^{p-2}\|\phi\|_{\sH^4}^2\leq  c\|\phi\|_{\sH^1}^{p-2}\left(\|\phi\|_{\sH^2}^2+|\Delta^2\phi|_2^2\right),\qquad p\geq2.
\]
Moreover, by \eqref{eq.ineq.E2}, \eqref{eq.ineq.E3}, we get that for some $c>0$, $p'\geq 2$ it holds
\[
     \|\phi_n\|_{\sH^{2+\frac{4}{p}}}^p\leq c(1+E(\phi_n)^{p'}) \left(1+\left|\frac{\delta E}{\delta \phi}(\phi_n)  \right|_2^2\right).
\]
Then, thanks to Theorem \ref{thm.moments}, we have that for any $p\geq2$ the sequence $(\phi_n)_n$ is uniformly bounded in $\sL^p\left(\Omega;\sL^p([0,T];\sH^{2+\frac{4}{p}})\right)$, $p\geq2$.

Consequently, by Proposition \ref{prop.holderphi} the sequence $(\phi_n)_n$ is   bounded in 
\[
  \sL^p\left(\Omega;\sL^p([0,T];\sH^{2+\frac{4}{p}})))\right)\ \cap\  \sL^2\left(\Omega;W^{\theta,p}([0,T];\sL^2(\qspace))\right)\  \cap \ \sL^2\left(\Omega;\sL^2([0,T];\sH^4(\qspace))\right)\qquad \theta<\frac12,  \,  p<\infty,
\]
endowed with the conditions $\phi_n=-1 $ on $\partial \qspace$, $\Delta\phi_n=0$ on $\partial \qspace$.
Since by \cite[Theorem 2.1 and Theorem 2.2]{FlandoliGatarek}) we have that the embeddings
\begin{eqnarray*} 
   &&  W^{\theta,p}([0,T];\sL^2(\qspace))\hookrightarrow \mathcal{C}([0,T];\sH^{-\sigma}(\qspace)),\qquad \sigma>0 \\
   && \sL^p([0,T];\sH^{2+\frac{4}{p}}(\qspace))\cap W^{\theta,p}([0,T];\sL^2(\qspace))\hookrightarrow \sL^p([0,T];\sH^2(\qspace)), \qquad  \theta p >2
\end{eqnarray*}
are compact, the result follows by Prokhorov's Theorem.
\end{proof}

\begin{Th} \label{th.compactness}
Let $(w_0,\phi_0)\in D(A)\times \sL^2(\qspace)$ with $\phi_0=-1$ on $\partial \qspace$.
Then, there exists a probability space $(\tilde\Omega,\mathcal{\tilde F},\tilde\Prb)$, 
two cylindrical Wiener processes $\tilde W(t)$, $\tilde Z(t)$ defined on $(\tilde\Omega,\mathcal{\tilde F},\tilde\Prb)$,  stochastic processes 
\begin{eqnarray*}
  && w\in    \mathcal{C}([0,T];D(A^{-\rho}))\cap \sL^p([0,T];\sV)\cap \sL^2([0,T];D(A)),\qquad \rho>0,\\ 
  && \phi\in \mathcal{C}([0,T];\sH^{-\sigma}(\qspace))\cap \sL^p([0,T];\sH^{2}(\qspace))\cap \sL^2([0,T];\sH^4(\qspace)), \qquad  \sigma>0, \\ 
  && \zeta\in \sL^2([0,T];\sL^2(\qspace))
\end{eqnarray*}
 and subsequences ( for simplicity they are not relabeled ) such that for any  $p<\infty$  and $\tilde\Prb$-a.s. the solution $(w_n,\phi_n)$ of problem \eqref{eq.approx}
 with $\tilde W(t)$ and $\tilde Z(t)$ instead of $W(t)$, $Z(t)$ satisfies
\begin{equation*}
 \begin{split}
 (i)\ & w_n\to w \quad \text{ strongly in } \mathcal{C}([0,T];D(A^{-\rho})),\, \rho>0\\ 
   (ii)\ & w_n\to w \quad \text{ strongly in } \sL^p([0,T]; \sV),\, p\in [1,\infty[\\
 (iii)\ & w_n\to w \quad \text{ weakly in } \sL^2([0,T]; D(A))\\
    (iv)\ & \phi_n\to\phi \quad \text{ strongly in } \mathcal{C}([0,T];\sH^{-\sigma}(\qspace)), \, \sigma>0 \\
  (v)\ & \phi_n\to\phi \quad \text{ strongly in } \sL^p([0,T];\sH^2),\, p\in[1,\infty[ \\
(vi)\ & \Delta^2\phi_n\to \Delta^2 \phi \text{ weakly in } \sL^2([0,T]; \sL^2(\qspace))\\
 (vii)\ & \dfrac{\delta E(\phi_{n})}{\delta \phi} \to \zeta \quad \text{ weakly in } \sL^2([0,T]; \sL^2(\qspace))\\
 (viii)\ & f(\phi_n)\to f(\phi) \text{  strongly  in } \sL^2([0,T]; \sL^2(\qspace))
\end{split}
\end{equation*}
\end{Th}

\begin{proof}
 
Taking into account Lemma \eqref{le.tightness}, by Skorohod representation theorem and by a diagonal extraction argument, there exists a probability space
$(\tilde\Omega,\mathcal{\tilde F},\tilde\Prb)$,  two cylindrical Wiener processes $\tilde W(t)$, $\tilde Z(t)$ defined on $(\tilde\Omega,\mathcal{\tilde F},\tilde\Prb)$, 
two stochastic processes $w,\phi$ such that the convergence conditions in  {\em (i)--(vi)} hold.

{\em (vii)}. By Theorem \ref{thm.moments}, the sequence $\dfrac{\delta  E(\phi_n)}{\delta \phi}$ are bounded in
$\sL^2(\Omega;\sL^2([0,T];\sL^2(\qspace)))$.
Then, by arguing as for the previous point, the result follows by Prokhorov theorem and Skorohod theorem.

$(viii)$ By the expression of $f(\phi_n)$ it is sufficient to show that   that $\Prb$-almost surely $\Delta\phi_n\to \Delta\phi$ 
 and $\phi_n^3\to \phi^3$ strongly in $\sL^p([0,T];\sL^2(\qspace))$. 
 Indeed, the two limits follows by {\em (v)} and by standard Sobolev embedding results.
\end{proof}

\section{ Proof ot Theorem  \ref{thm.intro}}
\subsection{Existence}
By Theorem \ref{th.compactness} we know that there exist subsequences $(w_n)_n$, $(\phi_n)_n$ converging $\widetilde \Prb$-a.s. to processes $(w,\phi)\in  \sL^2(\Omega;\sL^2([0,T];D(A)))\times \sL^2(\Omega;\sL^2([0,T];\sH^2(\qspace)))$.

The rest of the proof will be splitted in several lemma :
in Lemma \ref{le.supbound}, we will show that the processes $(w,\phi)$ satisfied \eqref{eq.aNSCHestimate}.
Then we will show that $w,\phi$ fulfill the definition \ref{def.sol} of a solution for the abstract problem.

\begin{Le} \label{le.supbound}
 Under hypothesis of Theorem \ref{thm.intro}, we have that \eqref{eq.aNSCHestimate} hold.
\end{Le}

\begin{proof}

Let us show the first bound of \eqref{eq.aNSCHestimate}. 
Let us notice that by the definition of the norm in $D(A^{\rho})$ it holds $\|w\|_{D(A^{-\rho})}\leq \|w\|_{\sH}$, for all $\rho>0$.
By Theorem \ref{th.compactness}, 
\[
   \sup_{t\in[0,T]}\|w(t)\|_{D(A^{-\rho})}= \lim_{n\to\infty}\left(\sup_{t\in[0,T]}\|w_n(t)\|_{D(A^{-\rho})}  \right)\leq \liminf_{n\to\infty}\left(\sup_{t\in[0,T]}\|w_n(t)\|_{\sH}  \right)
\]
By Fatou's lemma and Theorem \ref{thm.montentsup} we deduce that for any $k>0$ there exists $c>0$ depending on $k,T,w_0,\phi_0$ such that
\[
   \widetilde\E\left[  \sup_{t\in[0,T]}\|w\|_{\sH}^k\right]
    \leq    \liminf_{n\to\infty}\widetilde\E\left[\sup_{t\in[0,T]}\|w_n\|_{\sH}^k\right]  
    \leq c.
\]
With a similar argument it can be shown that for any $k>0$ there exists $c>0$ depending on $k,T,w_0,\phi_0$ such that
\[
   \widetilde\E\left[ \sup_{t\in[0,T]}|\phi|_2^k\right]
    \leq    \liminf_{n\to\infty}\widetilde\E\left[ \sup_{t\in[0,T]}|\phi_n|_2^k\right]  
    \leq c
\]
which implies that the first bound  in  \eqref{eq.aNSCHestimate} holds.
Let us show the second bound.
Notice that by Theorem \ref{th.compactness} we have, $\widetilde \Prb$-a.s., that the limit 
 $\left(\|w_n(t)\|_{\sV}+\|\phi_n(t)\|_{\sH^2}\right)\wedge M\to \left(\|w(t)\|_{\sV}+\|\phi(t)\|_{\sH^2}\right)\wedge M$ holds in $L^p([0,T])$, for all $p\geq1$ and $M>0$.
Then, by Lemma \ref{Le.convE} we have that the limit  
\[
    \lim_{n\to \infty} \left( \left(\|w_n(t)\|_{\sV}+\|\phi_n(t)\|_{\sH^2}\right)^p\wedge M\right)  \dfrac{\delta E(\phi_n(t))}{\delta \phi} 
    = \left(\left(\|w(t)\|_{\sV}+\|\phi(t)\|_{\sH^2}\right)^p\wedge M\right)  \dfrac{\delta E(\phi(t))}{\delta \phi}
\]
holds weakly in $L^2([0,T]\times \qspace)$, for any $M>0$. Then, for any $M>0$,
\begin{multline*}
    \int_0^T\left(\left(\|w(t)\|_{\sV}+\|\phi(t)\|_{\sH^2}\right)^{2p}\wedge M^2\right)  \left|\dfrac{\delta E(\phi(t))}{\delta \phi}\right|_2^2\dd t \\
       \leq \liminf_{n\to\infty}\int_0^T\left(\left(\|w_n(t)\|_{\sV}+\|\phi_n(t)\|_{\sH^2}\right)^{2p}\wedge M^2\right)  \left|\dfrac{\delta E(\phi_n(t))}{\delta \phi}\right|_2^2\dd t
\end{multline*}
Letting $M\to\infty$, by monotone convergence we obtain
\[
    \int_0^T\left(\|w(t)\|_{\sV}+\|\phi(t)\|_{\sH^2}\right)^{2p} \left|\dfrac{\delta E(\phi(t))}{\delta \phi}\right|_2^2\dd t 
       \leq \liminf_{n\to\infty}\int_0^T\left(\|w_n(t)\|_{\sV}+\|\phi_n(t)\|_{\sH^2}\right)^{2p}   \left|\dfrac{\delta E(\phi_n(t))}{\delta \phi}\right|_2^2\dd t
\]
Finally, by Fatou's Lemma we get %
\begin{multline*}
  \widetilde\E\left[ \int_0^T\left(|w(t)|_V+|\phi(t)|_{\sH^2}\right)^{2p}   \left|\dfrac{\delta E(\phi(t))}{\delta \phi}\right|_2^2\dd t \right]\\
       \leq \liminf_{n\to\infty}\widetilde\E\left[ \int_0^T\left(\|w_n(t)\|_{\sV}+\|\phi_n(t)\|_{\sH^2}\right)^{2p}  \left|\dfrac{\delta E(\phi_n(t))}{\delta \phi}\right|_2^2\dd t\right] \leq c
\end{multline*}
where $c>0$ is given by Theorem \ref{thm.moments}.
By similar arguments we can show that there exists $c>0$ such that
\[
    \widetilde\E\left[ \int_0^T \left(\|w(t)\|_{\sV}+\|\phi(t)\|_{\sH^2}\right)^{2p}\left(|\nabla w|_2^2 +\alpha^2|A  w|_2^2|\right)\dd t \right] \leq c.
\]
To conclude the proof, it is sufficient to notice that thanks to  \eqref{eq.ineq.E2bis} there exists  $c>0$ such that $E(\phi)\leq c(1+|\phi(t)|_{\sH^2}^8)$.
\end{proof}

\begin{Le} \label{le.def.sol}
 Under hypothesis of Theorem \ref{thm.intro}, the limit processes $(w,\phi)$ solve \eqref{eq.aNSCHabstract} in the sense of Definition \ref{def.sol}
\end{Le}
\begin{proof}
Let us first show that $(w,\phi)$ solve \eqref{eq.aNSCHabstract}.
Since $w_n$, $\phi_n$ solves \eqref{eq.approx}, it is sufficient to show that the right-hand side of \eqref{eq.approx} converges to the right-hand side of \eqref{eq.aNSCHabstract}.

Let $\xi \in \sL^2([0,T]; D(A)) $.  
By Theorem \ref{th.compactness},  {\em (iii)} we have 
\[
  \lim_{n\to\infty}\int_0^T\la   w_{n} + \alpha^{2}Aw_{n},\xi(t)\ra \dd t= \int_0^T\la   w + \alpha^{2}Aw,\xi(t)\ra \dd t 
\]
and
\begin{eqnarray*}
\lim_{n\to\infty }\nu  \int_{0}^{T} \langle\int_{0}^{t}( w_{n}(\tau) + \alpha^{2}Aw_{n}(\tau))\dd\tau,A\xi(t)\rangle \ddt
   &=&\lim_{n\to\infty } \nu \int_{0}^{t}\int_{0}^{t}\langle  w(\tau) + \alpha^{2}Aw(\tau), A\xi(t) \rangle d\tau \ddt \\
   &=&  \nu  \int_{0}^{T} \langle\int_{0}^{t}( w (\tau) + \alpha^{2}Aw (\tau))\dd\tau,A\xi(t)\rangle \ddt.
\end{eqnarray*}
Observe that by Proposition \ref{prop.1.1} (ii) it holds
\[
   \left|\int_0^T\int_0^t\la \widetilde{B}(w(\tau),u(\tau)),\xi(t)\ra \dd\tau\dd t\right|
     \leq \left(\int_0^T|w(\tau)|_V^2\dd \tau \right)^{\frac12}\left(\int_0^T|u(\tau)|_2^2\dd \tau \right)^{\frac12}\left(\int_0^T|\xi(t)|_{D(A)}^2\dd t \right)^{\frac12}.
\]
This implies that the trilinear form 
\begin{eqnarray*}
 && \sL^2([0,T];V)\times  \sL^2([0,T];\sL^2(\qspace))\times  \sL^2([0,T];D(A)) \to \R \\
 && (w,u,\xi)\mapsto \int_0^T \int_0^t\la \widetilde{B}(w(\tau),u(\tau)),\xi(t)\ra\dd\tau \dd t
\end{eqnarray*}
is continuous.
Since by Theorem \ref{th.compactness} we have that $\Prb$-a.s. $w_n\to w$ strongly in $\sL^2([0,T];\sV)$, that $w_{n} + \alpha^{2}Aw_{n}\to w + \alpha^{2}Aw $ weakly in $\sL^2([0,T];\sL^2(\qspace))$ and clearly $P_n\xi\to \xi $ strongly in $\sL^2([0,T];D(A))$, 
we deduce that
\begin{eqnarray*}
  &&\lim_{n\to\infty } \int_{0}^{T} \la \int_0^t   P_n\widetilde{B}(w_n  ,w_{n} + \alpha^{2}Aw_{n})(\tau)\dd\tau,\xi(t)\ra_{(D(A)',D(A))}\dd t
\\
 &&\qquad  = \lim_{n\to\infty } \int_{0}^{T} \int_0^t \la  \widetilde{B}(w_n  ,w_{n} + \alpha^{2}Aw_{n})(\tau),P_n\xi(t)\ra_{(D(A)',D(A))} \dd\tau \dd t
\\
 &&\qquad =  \int_{0}^{T} \int_0^t \la  \widetilde{B}(w  , w + \alpha^{2}Aw )(\tau),\xi(t)\ra_{(D(A)',D(A))} \dd\tau \dd t
\end{eqnarray*}
as $n\to\infty$.
Finally, it is easy to see that $\tilde \Prb$-a.s. it holds 
\[
   \lim_{n\to\infty} \int_0^T\la  \int_0^t(P_n \Ca)\dd\tilde W(\tau), \xi(t)\ra \dd t =  \int_0^T\la  \int_0^t  \Ca \dd\tilde  W(\tau), \xi(t)\ra \dd t.
\]
In order to complete the proof, we need the following
\begin{Le} \label{Le.convE}
 We have, $\widetilde \Prb$-a.s.
 \[
    \lim_{n\to\infty}   \dfrac{\delta E(\phi_{n}(t))}{\delta \phi}  = \dfrac{\delta E(\phi(t))}{\delta \phi} \qquad \text{weakly in } \sL^2([0,T];\sL^2(\qspace))
 \]
 and
 \[
    \lim_{n\to\infty} P_n\left( \dfrac{\delta E(\phi_{n}(t))}{\delta \phi}\right) = \dfrac{\delta E(\phi(t))}{\delta \phi} \qquad \text{weakly in } \sL^2([0,T];\sL^2(\qspace)).
 \]
\end{Le}
\begin{proof}
Let us prove the first limit.
By $(vii)$ of Theorem \ref{th.compactness} we have to show that $\zeta =  \dfrac{\delta E(\phi(t))}{\delta \phi}$.
Let $g\in \mathcal{C}_0^\infty([0,T]\times \qspace;\R)$.
We shall show that
\[
  \lim_{n\to\infty} \int_0^T\la \dfrac{\delta E(\phi_{n}(t))}{\delta \phi} ,g(t)\ra \dd t  
    =  \int_0^T\la \dfrac{\delta E(\phi(t))}{\delta \phi},g(t)\ra \dd t.
\]
By the expression \eqref{eq.varE2}  of $\dfrac{\delta E(\phi_{n})}{\delta \phi} $ we have to identify each limit. 
Indeed, if  we have
\[
   \int_0^T\la \Delta^2\phi_n(t),g(t)\ra \dd t\to   \int_0^T\la \Delta^2\phi (t),g(t)\ra \dd t 
\]
by $(vi)$ of Theorem \ref{th.compactness}.
Similarly,
\begin{eqnarray*}
    \lim_{n\to\infty} \int_0^T\la \Delta(\phi_n^3(t)-\phi_n(t)),g(t)\ra \dd t =   \int_0^T\la \phi_n^3(t)-\phi_n(t),\Delta g(t)\ra \dd t 
\\
  =\int_0^T\la \phi^3(t)-\phi(t),\Delta g(t)\ra \dd t  = \int_0^T\la \Delta(\phi^3(t)-\phi(t)),g(t)\ra \dd t 
\end{eqnarray*}
thanks to {\em (v)} of Theorem \ref{th.compactness}.
Moveover, by Theorem \ref{th.compactness}, {\em (v)}, {\em (viii)}, the limit  
\[
   \lim_{n\to\infty}  \int_0^T\la (3\phi^3_n-1)f(\phi_n) , g(t)\ra \dd t  = \int_0^T\la (3\phi^3 -1)f(\phi),g(t)\ra \dd t 
\]
holds. 
For the last term, we have to show that
\begin{equation} \label{eq.bfconvergence}
     \lim_{n\to\infty}  \int_0^T\mathcal{B} (\phi_n(t)) \la  f(\phi_n(t)) , g(t)\ra \dd t  = \int_0^T\mathcal{B} (\phi_n(t))\la  f(\phi(t)),g(t)\ra \dd t 
\end{equation}
Since $  \mathcal{B} (\phi_n)=\dfrac{1}{2}|\nabla\phi_n|_2^2+\dfrac{1}{4}|\phi_n^2-1|_2^2$,
by {\em (v)}  of Theorem \ref{th.compactness} we deduce that $\mathcal{B} (\phi_n)\to \mathcal{B} (\phi)$ in $\sL^p([0,T];\R)$, for any $p\in [1,\infty[$.
On the other side, by {\em (v)} of Theorem \ref{th.compactness} we have   $ f(\phi_n)g\to f(\phi)g$ as $n\to \infty$ in $\sL^p([0,T];\R)$. 
Then, we deduce that \eqref{eq.bfconvergence} holds.

The second limit is obvious since for any $g\in \mathcal{C}_0^\infty([0,T]\times \qspace;\R)$, $P_ng\to g$ strongly in $\sL^2([0,T];\sL^2(\qspace))$ and then
\begin{eqnarray*}
     \lim_{n\to\infty} \int_0^T \la P_n\left( \dfrac{\delta E(\phi_{n}(t))}{\delta \phi}\right),g(t)\ra\ddt &=& \lim_{n\to\infty} \int_0^T \la   \dfrac{\delta E(\phi_{n}(t))}{\delta \phi} ,P_ng(t)\ra\ddt    \\
     &=&\int_0^T \la   \dfrac{\delta E(\phi(t))}{\delta \phi} , g(t)\ra\ddt. 
\end{eqnarray*}
\end{proof}
By the previous lemma and by {\em (vii)} of Theorem \ref{th.compactness} we get that 
\[ 
 \int_0^t P_n\left( \dfrac{\delta E(\phi_{n}(\tau))}{\delta \phi}\right)\nabla\phi_n(\tau)\dd\tau \to \int_0^t  \dfrac{\delta E(\phi (\tau))}{\delta \phi} \nabla\phi (\tau)\dd\tau
\]
strongly in $\sL^2([0,T];\sL^2(\qspace))$. 
Then, in particular, the convergence  holds weakly in $\sL^2([0,T];D(A))$.
So, we have show that $w$ solves the first equation of \eqref{eq.aNSCHabstract}. 
Let us show that $\phi$ solve the second one.
Let us observe that by Theorem \ref{th.compactness}, $\phi_n\to \phi$ strongly in $\sL^p([0,T];\sH^2(\qspace))$.
Moreover, since $w_n\to w$ strongly in $\sL^p([0,T];\sV)$, it easy to show that the limit
\[
  \lim_{n\to\infty} \int_0^t \pi_n\left(w_n\nabla\phi_n  \right)(\tau)\dd\tau    =  \int_0^t (w\nabla\phi_n)(\tau)\dd\tau 
\]
holds in $\sL^2([0,T];\sL^2(\qspace))$.
Finally, it is clear that 
\[
   \lim_{n\to\infty} \int_0^T\la  \int_0^t(\pi_n \Cb)\dd\tilde Z(\tau), \xi(t)\ra \dd t =  \int_0^T\la  \int_0^t  \Cb \dd\tilde  Z(\tau), \xi(t)\ra \dd t.
\]
holds $\tilde \Prb$-a.s.
Then, $(w,\phi)$ is a solution of \eqref{eq.aNSCHabstract}.\\
It remains to verify that $(w,\phi)$ satisfy all the other conditions of Definition \ref{def.sol}.
{\em Continuity of $w+\alpha^2 Aw$, $\phi$.} Notice that since $w+\alpha^2 Aw$, $\phi$ solves  the stochastic differential equation \eqref{eq.def.sol}, then
$w+\alpha^2 Aw \in \sL^2(\Omega;\mathcal{C}([0,T]; D(A)'))$ and $\phi\in \sL^2(\Omega;\mathcal{C}([0,T]; \sL^2(\qspace)))$.
The fact that $\phi$, $w$ are adapted to the filtration $\mathcal{F}_{t}$ is obvious, been $\phi$, $w$ a.s. limit of adapted processes.

It remains to show that $\phi$, $w$ are continuous in mean square.
Indeed, by Itô formula \ref{eq.itow_n} we deduce, 
\begin{eqnarray*}
 && \E\left[|w_n(t)-w_n(t_0)|_2^2+\alpha^2|\nabla (w_n(t)-w_n(t_0))|_2^2\right] \leq  \E\int_{t_0}^t\nu(|\nabla w_n|_2^2 +\alpha^2|A  w_n|_2^2)\dd s\\
  &&\qquad   +\E\int_{t_0}^t\left| \dfrac{\delta E(\phi_n)}{\delta \phi}\right|_2\left|\nabla \phi_n\right|_2|w_n|_\infty \dd s
                                                    +  \frac12 \Tr[(\pi_n\Ca)^*(I+\alpha^2A)^{-1} (\pi_n\Ca)](t-t_0)\\
 && \leq  \E\int_{t_0}^t\nu(|\nabla w_n|_2^2 +\alpha^2|A  w_n|_2^2)\dd s\\
  &&\qquad   +\E\int_{t_0}^t\left| \dfrac{\delta E(\phi_n)}{\delta \phi}\right|_2\left|\nabla \phi_n\right|_2|w_n|_V \dd s
                                                    +  \frac12 \Tr[(\Ca^*(I+\alpha^2A)^{-1} \Ca](t-t_0). 
\end{eqnarray*}
Notice that we have used the property $\la \tilde B(w,w+\alpha^2Aw),w\ra=0$. 
Moreover, by  Theorem \ref{th.compactness} and the bounds in \eqref{eq.aNSCHestimate} we can apply Fatou's Lemma to get, as $n\to\infty$
\begin{eqnarray*}
 && \E\left[|w(t)-w(t_0)|_2^2+\alpha^2|\nabla (w(t)-w(t_0))|_2^2\right] \leq  \E\int_{t_0}^t\nu(|\nabla w|_2^2 +\alpha^2|A  w|_2^2)\dd s\\
  &&\qquad   +\E\int_{t_0}^t\left| \dfrac{\delta E(\phi)}{\delta \phi}\right|_2\left|\nabla \phi\right|_2|w|_V \dd s
                                                    +  \frac12 \Tr[(\Ca^*(I+\alpha^2A)^{-1} \Ca](t-t_0). 
\end{eqnarray*}
Then, the continuity in mean square for $w$ follows. 
In a similar way (we omit the calculus, which are standard) we get the continuity in mean square for the process $\phi$.
\end{proof}

\begin{Coro}
Under hypothesis of Theorem \ref{thm.intro}, we have
\[
     \lim_{n\to\infty} \int_0^t P_n\left( \dfrac{\delta E(\phi_{n}(\tau))}{\delta \phi}\right)\dd \tau  = \int_0^t\dfrac{\delta E(\phi(\tau))}{\delta \phi}\dd \tau  \qquad \text{ in } \sL^p([0,T];\sL^2(\qspace)),\, p\in [1,\infty[.
\]
\end{Coro}

\subsection{Uniqueness}
\begin{Th}
 Under Hypothesis \ref{hyp.trace} for any initial condition $(w_0,\phi_0)\in D(A)\times \sL^2(\qspace)$ 
  there exists a unique solution $(w,\phi)$ to equation \eqref{eq.aNSPF} such that for any $T>0$ and $\Prb$-a.s.
  \begin{equation}  \label{eq.cond.uniqueness}
  \int_0^T\left(|w(t)|_V^2+ \left|\frac{\delta E}{\delta \phi}(\phi(t))\right|_2^2+|\phi|_{\sH^2}^8+|\Delta^2\phi|_2^2\right)\dd t<\infty
  \end{equation}
\end{Th}

Since the the proof of this result is quite the same as in \cite{EntringerBoldrini},
for the reader's convenience we only give here the main ideas.
\begin{proof}
By Theorem \ref{th.compactness} and Theorem \ref{le.supbound}, there exists at least a solution $(\omega,\phi)$ satisfying \eqref{eq.cond.uniqueness}.
As usual, consider two solutions of the system $(w_{1},\phi_{1})$ and $(w_{2},\phi_{2})$ with the expected regularity stated before, 
and consider the difference  $(w, \phi)=(w_1,\phi_1)-(w_2,\phi_2)$ between these two solutions.
We shall show that $(w_1,\phi_1)=(w_2,\phi_2)$ on the full measure set  
\begin{equation} \label{eq.cond.uniqueness2}
 \left\{  \sum_{i=1}^2 \int_0^T\left(|w_i(t)|_V^2+ \left|\frac{\delta E}{\delta \phi}(\phi_i(t))\right|_2^2+|\phi_i|_{\sH^2}^8+|\Delta^2\phi_i|_2^2\right)\dd t<\infty\right\}.
\end{equation}
As in \cite{EntringerBoldrini}, we write
\[
  \dfrac{\delta E(\phi)}{\delta \phi}(\phi)=M(\phi)+N(\phi),
\]
where
\[ 
 M(\phi)=\Delta^2\phi-\Delta\phi+\phi
\]
\[
  N(\phi)=\dfrac{\delta E}{\delta \phi}(\phi)-M(\phi).
\]
Let us set $G(\phi)= |\Delta \phi|_2^2+|\nabla\phi|_2^2+|\phi|_2^2$.
The proof of the following lemma is easy and it is left to the reader.
\begin{Le} \label{le.G}
The function $G(\phi)$ defines a norm equivalent to the $\sH^2(\qspace)$ norm. 
That is, there exists $C>0$ such that it holds
\[
  \frac{1}{C}\|\phi\|_{\sH^2}^2\leq G(\phi)\leq C \|\phi\|_{\sH^2}^2,\qquad  \forall \phi\in \sH^2(\qspace)
\]
Moreover,  it holds
\[
 \int_0^T G(\phi) \ddt   =    \int_0^T\la M(\phi),\phi\ra\ddt,\qquad \forall \phi\in \sL^2([0,T];\sH^4(\qspace))\cap \{ \phi: \phi=\Delta\phi=0 \text{ on } \partial \qspace\}.
\]
\end{Le}
For any $ \tilde v\in \sL^2([0,T];D(A))$, the couple $(w,\phi)$ satisfies
\begin{equation*}
\begin{cases}
\dd \la  w+\alpha^{2} A w,\tilde v\ra  =  \left(  \la -\nu A (w+\alpha^{2}Aw),\tilde v\ra + \la- \widetilde{B}(w_{1},w_{1}+\alpha^{2} Aw_{1})+\widetilde{B}(w_{2},w_{2}+\alpha^{2} Aw_{2}),\tilde v\ra \right.
\\
\qquad \left. +\la M( \phi_1),  \nabla \phi_1\cdot\tilde  v \ra - \la M( \phi_2), \nabla \phi_2\cdot\tilde  v \ra
 +\la N(\phi_1), \phi_1\cdot \tilde v\ra -\la N(\phi_2),\nabla\phi_2\cdot\tilde v\ra \right) \ddt & \text{in } [0,T]\times \qspace,\\
\dd  \phi  =  \left(-w_{1}\cdot \nabla \phi_{1}+w_{2}\cdot \nabla \phi_{2} 
-  M(\phi) -N(\phi_1)+N(\phi_2) \right)\ddt    
 & \text{in } [0,T]\times \qspace, \\
w(0)=0& \text{in } \qspace, \\
\phi(0)=0 & \text{in } \qspace. 
\end{cases}
\end{equation*}
Let us look at the second equation.
By multiplying with $M(\phi)$ and integrating over $[0,t]\times \qspace$ we find 
\begin{equation} \label{eq.phi.unique}
 \frac12  G(\phi(t)) = 
          -\int_0^t\la  w_{1}\cdot \nabla \phi_{1}-w_{2}\cdot \nabla \phi_{2},M(\phi)\ra\dd s
       -\gamma \int_0^t\left(\la \dfrac{\delta E(\phi_{1})}{\delta \phi}- \dfrac{\delta E(\phi_{2})}{\delta \phi},M(\phi) \ra   \right)  \dd s
\end{equation}
Here we used the fact that $E(\phi_i)<\infty$ implies $\phi_i\in \sH^2(\qspace)$ ( see \eqref{eq.ineq.E2} ).
Moreover, notice that $\phi=\Delta\phi=0$ on $\partial \qspace$ and \eqref{eq.cond.uniqueness} holds,  then we can apply the integration by parts in Lemma \ref{le.G}.

Since $w_i\in  \sL^2([0,T];D(A))$, we can set $\tilde v=w$ in the first equation.
By integrating over $[0,t]\times \qspace$ we find 
\begin{equation} \label{eq.w.unique}
\begin{split}
   \frac12 (|w(t)|^2_2   &+  \alpha^2|\nabla w(t)|_2^2) +\nu\int_0^t\left( |\nabla w|_2^2+\alpha^{2}|\Delta w|^2_2\right)\dd s  \\
    =&  \int_0^t\left(\la -\widetilde{B}(w_{1},w_{1}+\alpha^{2} Aw_{1})+\widetilde{B}(w_{2},w_{2}+\alpha^{2} Aw_{2}) ,w\ra \right)\dd s \\
     & +\int_0^t\left(\la \dfrac{\delta E(\phi_{1})}{\delta \phi}\nabla \phi_{1}-\dfrac{\delta E(\phi_{2})}{\delta \phi}\nabla \phi_{2},w\ra\right) \dd s  \\
    =&   \int_0^t\la \widetilde{B}(w,w+\alpha^2Aw),w_2\ra\dd s
         +\int_0^t\left(\la M(\phi_1)\nabla \phi_{1}-M(\phi_2)\nabla \phi_{2},w\ra\right)\dd s\\
    &+\int_0^t \left(\la N(\phi_1)\nabla \phi_{1}-N(\phi_2)\nabla \phi_{2},w\ra\right)\dd s
\end{split}
\end{equation}
Here we have used the properties of $\widetilde{B}$ ( see Proposition \ref{prop.1.1} ) which yield
\[
  \la \widetilde{B}(w_{1},w_{1}+\alpha^{2} Aw_{1}),w\ra-\la\widetilde{B}(w_{2},w_{2}+\alpha^{2} Aw_{2}, w\ra =-\la \widetilde{B}(w,w+\alpha^2Aw),w_2\ra.
\]
By adding \eqref{eq.w.unique} and \eqref{eq.phi.unique} we get 
\[
   \frac12 (|w(t)|^2_2+\alpha^2|\nabla w(t)|_2^2+|G(\phi(t))|_2^2) +\nu\int_0^t\left(|\nabla w|^2+\alpha^{2}|\Delta w|^2_2\right)\dd s
    +\gamma \int_0^t|M(\phi)|_2^2\dd s =  \int_0^t\mathcal{F}(w,\phi) \dd s 
\]
where
\begin{eqnarray*}
  \mathcal{F}(w,\phi)    &=&   \la \widetilde{B}(w,w+\alpha^2Aw),w_2\ra +\la M(\phi_1)\nabla \phi_{1}-M(\phi_2)\nabla \phi_{2},w\ra\\
    &&-\la  w_{1}\cdot \nabla \phi_{1}-w_{2}\cdot \nabla \phi_{2},M(\phi)\ra+\la N(\phi_1)\nabla \phi_{1}-N(\phi_2)\nabla \phi_{2},w\ra\\
    &&  -\gamma\la N(\phi_1)- N(\phi_2),M(\phi)\ra 
\end{eqnarray*}
As in \cite{EntringerBoldrini}, we have to estimate each term of $\mathcal{F}(w,\phi)$. 
A key role is played by the following result, which is similar to Lemma 5.2 of \cite{EntringerBoldrini}.
The main difference is that in \cite{EntringerBoldrini} the solution $\phi$ belongs to ${\mathcal C}^0([0,T];\sH^2)$.
In our case, we are able to prove only $\phi\in L^p([0,T];\sH^2)$.   
\begin{Le} \label{lemma.boundN}
 Let $\phi_1$, $\phi_2\in \sH^2(\qspace)$ such that $\phi_i+1=\Delta\phi_i=0$ on $\partial \qspace$, $i=1,2$. 
 Then there exists 
 $c>0$, independent by $\phi_1$, $\phi_2$ such that 
\begin{equation}  \label{eq.boundN}
  |N(\phi_1)-N(\phi_2)|_2\leq c\left(1+\|\phi_1\|_{\sH^2}^6+\|\phi_2\|_{\sH^2}^6\right)\|\phi_1-\phi_2\|_{\sH^2}.
\end{equation}
\end{Le}
\begin{proof}
By \eqref{eq.varE2} we can write
\[
   N(\phi)=  -\Delta\phi^3+2\Delta \phi +3\phi^2f(\phi)-f(\phi)-\phi +M_1(\mathcal{A}(\phi)-a)+M_2(\mathcal{B}(\phi)-b)f(\phi).
\]
Then,
\begin{eqnarray*}
|  N(\phi_1)- N(\phi_2)|_2&\leq& |\Delta(\phi_1^3-\phi_2^3)|_2+2|\Delta (\phi_1-\phi_2)|_2 +3|\phi_1^2f(\phi_1)-\phi_2^2f(\phi_2)|_2+\\
  &&  +(1+M_2b)|f(\phi_1)-f(\phi_2)|_2+ |\phi_1-\phi_2|_2 +M_1|\mathcal{A}(\phi_1)-\mathcal{A}(\phi_1)|_2+\\
  && +M_2|\mathcal{B}(\phi_1)f(\phi_1)-\mathcal{B}(\phi_2)f(\phi_2)|_2=I_1+\ldots+I_7.
\end{eqnarray*}
Let us proceed by estimating each term. 
For $I_1$, we set $\tilde \phi=\phi_1^2+\phi_1\phi_2+\phi_2^2$. .
Using Poincaré inequality, it holds $|\phi_i|_\infty\leq (1+\Cp|\phi_i|_{\sH^1})$, $i=1,2$ where $\Cp>0$ is the Poincaré constant. 
We deduce that there exists $c>0$ such that
\begin{equation} \label{eq.tildephi1}
   |\tilde \phi|_\infty\leq  \sum_{i,j=1}^2|\phi_i|_\infty|\phi_j|_\infty \leq \sum_{i,j=1}^2(1+\Cp\|\phi_i\|_{\sH^1})(1+\Cp\|\phi_j\|_{\sH^1}) \leq c(1+\|\phi_1\|_{\sH^1}^2+\|\phi_2\|_{\sH^1}^2).
\end{equation}
Similarly, since $|\nabla \tilde \phi|\leq   \sum_{i,j=1}^2|\phi_i||\nabla\phi_j|$ by Poincaré inequality and the Sobolev embedding $\sH^1(Q)\subset \sL^4(Q)$ there exists $c>0$ such that
\begin{eqnarray} 
   |\nabla \tilde \phi|_4 &\leq& \sum_{i,j=1}^2|\phi_i|_\infty |\nabla \phi_j|_4 \leq \sum_{i,j=1}^2(1+\Cp\|\nabla\phi_i\|_{\sH^1}) \|\nabla \phi_j\|_{\sH^1} \notag
\\
  &\leq& \sum_{i,j=1}^2(1+\Cp\|\phi_i\|_{\sH^2} )\|\phi_j\|_{\sH^2} \leq  (1+\|\phi_1\|_{\sH^2}^2+\|\phi_2\|_{\sH^2}^2).\label{eq.tildephi2}
\end{eqnarray}
Moreover, since $\Delta\tilde\phi = \sum_{\stackrel{i,j=1}{i\leq j}}^2\Delta(\phi_i\phi_j)= \sum_{\stackrel{i,j=1}{i\leq j}}^2\left((\Delta\phi_i)\phi_j+\nabla\phi_i\nabla\phi_j+\phi_i(\Delta\phi_j)\right)$,
still using the Poincaré inequality and the Sobolev embedding $\sH^1(\qspace)\subset \sL^4(\qspace)$  we get
\begin{eqnarray}
   |\Delta\tilde\phi|_2 &\leq&  \sum_{\stackrel{i,j=1}{i\leq j}}^2\left(|\Delta\phi_i|_2|\phi_j|_\infty+|\nabla\phi_i|_4|\nabla\phi_j|_4+|\phi_i|_\infty|\Delta\phi_j|_2\right)\notag
\\
  &\leq& \sum_{\stackrel{i,j=1}{i\leq j}}^2\left(|\Delta\phi_i|_2(\Cp\|\phi_j\|_{\sH^1}+1)+\|\phi_i\|_{\sH^2}\|\phi_j\|_{\sH^2}+(\Cp\|\phi_i\|_{\sH^1}+1)|\Delta\phi_j|_2\right)\notag
\\
   &\leq& c(1+\|\phi_1\|^2_{\sH^2}+\|\phi_2\|^2_{\sH^2})\label{eq.tildephi3}
\end{eqnarray}
where $c>0$ is independent by $\phi_1$, $\phi_2$.
By taking in mind \eqref{eq.tildephi1},\eqref{eq.tildephi2}, \eqref{eq.tildephi3} there exists $c>0$ such that

\begin{eqnarray*}
I_1=|\Delta(\phi_1^3-\phi_2^3)|_2&=& |\Delta(\phi  \tilde \phi)|_2 \\
                             &\leq& |(\Delta\phi)\tilde \phi|_2+ 2|\nabla\phi\cdot\nabla \tilde \phi|_2+| \phi(\Delta\tilde\phi)|_2\\
                             &\leq& |\Delta\phi|_2|\tilde \phi|_\infty+ 2|\nabla\phi|_4|\nabla \tilde \phi|_4+| \phi|_\infty|\Delta\tilde\phi|_2\\
                             &\leq& c(1+\|\phi_1\|^2_{\sH^2}+\|\phi_2\|^2_{\sH^2})(|\Delta\phi|_2+|\nabla\phi|_4+| \phi|_\infty)\\
                             &\leq& c(1+\|\phi_1\|^2_{\sH^2}+\|\phi_2\|^2_{\sH^2})\|\phi_1-\phi_2\|_{\sH^2}.
\end{eqnarray*}
In the last inequality we have used the Sobolev embedding $\sH^1(\qspace)\subset \sL^4(\qspace)$ and the Poincaré inequality $|\phi_1-\phi_2|_\infty\leq \Cp|\phi_1-\phi_2|_{\sH^1}$.
For $I_2$, we have clearly $I_2\leq c|\phi_1-\phi_2|_{\sH^2}$.
For $I_3$ we can write
\[
   I_3= 3|\phi_1^2f(\phi_1)-\phi_2^2f(\phi_2)|_2   \leq   3|\phi_1^2(f(\phi_1)-f(\phi_2))|_2+3|(\phi_1^2-\phi_2^2)f(\phi_2)|_2=J_1+J_2.
\]
For $J_1$, we have
\[
   J_1\leq 3|\phi_1|_\infty^2|f(\phi_1)-f(\phi_2))|_2\leq 3|\phi_1|_\infty^2(|\Delta(\phi_1-\phi_2)|_2+|\phi_1^3-\phi_2^3|_2+|\phi_1-\phi_2|_2).
\]
With a similar calculus done for $I_1$, we have $|\phi_1^3-\phi_2^3|_2\leq |\tilde \phi|_\infty|\phi_1-\phi_2|_2$. 
Then, using Poincaré inequality there exists $c>0$ such that $|\phi_1^3-\phi_2^3|_2\leq c(1+\|\phi_1\|_{\sH^1}^2+\|\phi_2\|_{\sH^1}^2)|\phi_1-\phi_2|_2$.
Then, for some $c>0$ independent by $\phi_1$, $\phi_2$ we obtain
\[
     J_1\leq c(1+\|\phi_1\|_{\sH^1}^2)(1+\|\phi_1\|_{\sH^1}^2+\|\phi_2\|_{\sH^1}^2)\|\phi_1-\phi_2\|_{\sH^2}.
\]
For $J_2$, since $H^1(\qspace)\subset \sL^6(\qspace)$ we have
\begin{eqnarray*}
  J_2&\leq& 3|\phi_1-\phi_2|_\infty |\phi_1+\phi_2|_\infty\left(|\Delta\phi_2|_2+|\phi_2|_6^3+|\phi_2|_2   \right) \\
      &\leq&c\|\phi_1-\phi_2\|_{\sH^1}(1+ \|\phi_1\|_{\sH^1}+\|\phi_2\|_{\sH^1})\left(\|\phi_2\|_{\sH^2}+\|\phi_2\|_{\sH^1}^3+|\phi_2|_2   \right) \\
\end{eqnarray*}
Finally, using Young inequality repeatedly, we find that for some $c>0$ it holds
\[
   I_3\leq c\left(1+|\phi_1|_{\sH^2}^6+|\phi_2|_{\sH^2}^6\right)|\phi_1-\phi_2|_{\sH^2}.
\]
For $I_4$, we can perform a calculus as done for $J_1$ to obtain, for some $c>0$
\[
     I_4=(1+M_2b)|f(\phi_1)-f(\phi_2)|_2 \leq c (1+\|\phi_1\|_{\sH^1}^2+\|\phi_2\|_{\sH^1}^2)\|\phi_1-\phi_2\|_{\sH^2}.
\]
Clearly, for $I_5$ and $I_6$ we have
\[
  I_5+I_6=|\phi_1-\phi_2|_2+M_1|A(\phi_1)-A(\phi_2)|_2\leq |\phi_1-\phi_2|_2+M_1|\qspace|^{1/2}|A(\phi_1-\phi_2)|\leq (1+M_1|Q|)|\phi_1-\phi_2|_2.
\]
For $I_7$,
\begin{eqnarray*}
 I_7&=&M_2|\mathcal{B}(\phi_1)f(\phi_1)-\mathcal{B}(\phi_2)f(\phi_2)|_2 \\
 &\leq& \mathcal{B}(\phi_1)|f(\phi_1)-f(\phi_1)|_2+|\mathcal{B}(\phi_1)-\mathcal{B}(\phi_2)||f(\phi_2)|_2\\
 &=& K_1+K_2
\end{eqnarray*}
For $K_1$, since $\mathcal{B}(\phi)=\frac12|\nabla\phi|_2^2+\frac14|\phi^2-1|_2^2$ we have
\[
   \mathcal{B}(\phi_1)   \leq \frac12|\nabla \phi_1|_{2}^2+\frac14(|\phi_1|_4^4+|\qspace|)
   \leq \frac12|\nabla \phi_1|_{2}^2+\frac14(\|\phi_1\|_{\sH^1}^4+|\qspace|).
\]
Then, there exists $c>0$ such that  $\mathcal{B}(\phi_1)\leq c(1+\|\phi_1\|_{\sH^1}^4)$.
In order to estimate $|f(\phi_1)-f(\phi_2)|_2$ we can argue as done before for the term $J_1$ to obtain 
\[
  |f(\phi_1)-f(\phi_2)|_2\leq c(1+\|\phi_1\|_{\sH^1}^2+\|\phi_2\|_{\sH^1}^2)\|\phi_1-\phi_2\|_{\sH^2}
\]
Then, by using Young inequality repeatedly, there exists $c>0$ such that
\[
    K_1\leq c(1+\|\phi_1\|_{\sH^1}^6+\|\phi_2\|_{\sH^1}^6) \|\phi_1-\phi_2\|_{\sH^2}.
\]
Before consider $K_2$, let us observe that by the expression of $\mathcal{B}(\phi_1)$ we have
\begin{eqnarray*}
    |\mathcal{B}(\phi_1) -\mathcal{B} (\phi_2)| &\leq&   \frac12\left|\la \nabla(\phi_1-\phi_2),\nabla(\phi_1+\phi_2)\ra\right| +\frac14\left|\la \phi_1-\phi_2,(\phi_1+\phi_2)(\phi_1^2+\phi_2^2-2)\ra  \right| \\
         &\leq&   \frac12\|\phi_1-\phi_2\|_{\sH^1}(\|\phi_1\|_{\sH^1}+\|\phi_2\|_{\sH^1}) +\frac14|\phi_1-\phi_2|_2|(\phi_1+\phi_2)(\phi_1^2+\phi_2^2-2)|_2.
\end{eqnarray*}
By Young inequality $ab^2\leq a^2/3+2b^3/3$ we get $(\phi_1+\phi_2)(\phi_1^2+\phi_2^2-2)\leq 2\phi_1^3+2\phi_2^3+\phi_1+\phi_2$.
Therefore, the last expression is bounded by 
\[
   |\mathcal{B}(\phi_1) -\mathcal{B} (\phi_2)| \leq \|\phi_1-\phi_2\|_{\sH^1}\left(\frac12\|\phi_1+\phi_2\|_{\sH^1} +\frac{|Q|^{1/2}}{4} \left(|\phi_1|_\infty^3+2|\phi_2|_\infty^3+|\phi_1|_2+|\phi_2|_\infty\right)\right).
\]
Since by Poincaré inequality we have $|\phi_i|_\infty\leq |\phi_i+1|_\infty+1\leq \Cp\|\phi_i\|_{\sH^1}+1$, we deduce that there exists $c>0$ such that
\[
   |\mathcal{B}(\phi_1) -\mathcal{B} (\phi_2)|\leq c\|\phi_1-\phi_2\|_{\sH^1}\left(1+\|\phi_1\|_{\sH^1}^3+2\|\phi_2\|_{\sH^1}^3\right)
\]
Moreover, since $f(\phi)=-\Delta\phi+\phi(\phi^2-1)$ and the continuous embedding $\sH^1(\qspace)\subset \sL^6(\qspace)$ holds, there exists $c>0$ such that 
\[
    |f(\phi_2)|_2\leq |\Delta\phi_2|_2+|\phi_2|_6^3+|\phi_2|_2\leq \|\phi_2\|_{\sH^2}+|\phi_2|_{\sH^1}^3+|\phi_2|_2\leq c(1+\|\phi_2\|_{\sH^2}^3)
\]
By the previous results, we deduce that for $K_2$ we have
\[
  K_2\leq c|\phi_1-\phi_2|_{\sH^1}\left(1+\|\phi_1\|_{\sH^1}^3+2\|\phi_2\|_{\sH^1}^3\right)(1+\|\phi_2\|_{\sH^2}^3)
\]
Then, for some $c>0$ independent by $\phi_1,\phi_2$ we obtain the bound
\[
   K_2\leq c|\phi_1-\phi_2|_{\sH^1}\left(1+\|\phi_1\|_{\sH^2}^6+|\phi_2|_{\sH^2}^6\right).
\]
Taking into account the estimates on $K_1$ and $K_2$ we get that for some $c>0$ we have
\[
  I_7\leq c\left(1+\|\phi_1\|_{\sH^2}^6+\|\phi_2\|_{\sH^2}^6\right)\|\phi_1-\phi_2\|_{\sH^2}.
\]
Finally, taking into account the estimates on $I_1,\ldots,I_7$, we get that there exits $c>0$ such that \eqref{eq.boundN} holds.
\end{proof}
By arguing as in \cite{EntringerBoldrini} (see equations (60)--(67)), the term $\mathcal{F}(w,\phi)$ is bounded by
\begin{eqnarray*}
  \mathcal{F}(w,\phi) &\leq& C_{\tilde \varepsilon}|\nabla w|_{\sH}^2|Aw_2|_{\sH}^2+ \tilde \varepsilon|w+\alpha^2Aw|_{\sH}^2 \\
		      && + \tilde \varepsilon\|w\|_{\sV}^2+C_{\tilde \varepsilon}|M(\phi_1)|_2^2\|\phi\|_{\sH^2}^2 \\
		      && + \tilde \varepsilon|M(\phi)|_2^2+C_{\tilde \varepsilon}\|w_1\|_{\sV}^2\|\phi\|_{\sH^2}^2 \\
		      && + \tilde \varepsilon\|w\|_{\sV}^2+C_{\tilde \varepsilon}|N(\phi_1)|_2^2\|\phi\|_{\sH^2}^2 \\
		      && + \tilde \varepsilon\|w\|_{\sV}^2+C_{\tilde \varepsilon}|N(\phi_1)-N(\phi_2)|_2^2\|\phi_2\|_{\sH^2}^2\\
		      && + \tilde \varepsilon|M(\phi)|_2^2+C_{\tilde \varepsilon}|N(\phi_1)-N(\phi_2)|_2^2,
\end{eqnarray*}
where $ \tilde \varepsilon>0$ can be chosen arbitrarly and $C_{\tilde \varepsilon}>0$ depends only on $\tilde \varepsilon>0$.
By \eqref{eq.boundN} there exists $c>0$ such that
\begin{eqnarray*}
  \mathcal{F}(w,\phi) &\leq&  3\tilde \varepsilon\|w\|_{\sV}^2+2\tilde \varepsilon|M(\phi)|_2^2+\tilde \varepsilon|w+\alpha^2Aw|_{\sH}^2+ C_{\tilde \varepsilon}|\nabla w|_{\sH}^2\|Aw_2\|_{\sH}^2\\
		       && +\left(C_{\tilde \varepsilon}(|M(\phi_1)|_2^2+\|w_1\|_{\sV}^2 +|N(\phi_1)|_2^2+c(1+\|\phi_2\|_{\sH^2}^2)(1+\|\phi_1\|_{\sH^2}^6+\|\phi_2\|_{\sH^2}^6)\right)\|\phi\|_{\sH^2}^2.
\end{eqnarray*}
Since $|N(\phi)|_2\leq |\dfrac{\delta E}{\delta \phi}(\phi)|_2+|M(\phi)|_2$ and $|M(\phi)|_2\leq c(|\Delta^2\phi|_2+\|\phi\|_{\sH^2})$, there exists $c_1,c_2>0$, depending only on $\tilde \varepsilon$ such that
  \begin{eqnarray*}
 \mathcal{F}(w,\phi) &\leq&  3\tilde \varepsilon\|w\|_{\sV}^2+2\tilde \varepsilon|M(\phi)|_2^2+\tilde \varepsilon\|w+\alpha^2Aw\|_{\sH}^2+ C_{\tilde \varepsilon}\|\nabla w\|_{\sH}^2\|Aw_2\|_{\sH}^2\\
		       && +c_1\left(|\Delta^2\phi_1|_2^2+\|w_1\|_{\sV}^2 +\left|\dfrac{\delta E}{\delta \phi}(\phi_1)\right|_2^2+1+\|\phi_1\|_{\sH^2}^8+\|\phi_2\|_{\sH^2}^8\right)\|\phi\|_{\sH^2}^2\\
		      &\leq& 3\tilde \varepsilon\|w\|_{\sV}^2+2\tilde \varepsilon|M(\phi)|_2^2+\tilde \varepsilon\|w+\alpha^2Aw\|_{\sH}^2\\
		      && + c_2\left(|\Delta^2\phi_1|_2^2+\|w_1\|_{\sV}^2 +\left|\dfrac{\delta E}{\delta \phi}(\phi_1)\right|_2^2+\|\phi_1\|_{\sH^2}^8+\|\phi_2\|_{\sH^2}^8+1\right)(G(\phi)+\|w\|_{\sH}^2+\alpha^2\|\nabla w\|_{\sH}^2)
\end{eqnarray*} 
Consequently, for $\tilde \varepsilon$ small enough, 
it holds 
\begin{multline*}
 \frac12 \left(|w(t)|_2^{2}+ \alpha ^{2}|\nabla w(t)|_2^{2} + G(\phi) (t)\right)
+ \frac{\gamma}{2} \int_{0}^{T}  |M(\phi)|_2^{2}   \dd t + \frac{\nu}{2}  \int_{0}^{T}   \left(|\nabla w(t)|_2^{2}+ \alpha ^{2}|A w(t)|_2^{2}\right)  \ddt\\
\leq\frac12 \int_{0}^{T} H(t)\left(   |w(t)|_2^{2}+ \alpha ^{2}|\nabla w(t)|_2^{2}   + G(\phi) (t)\right)\dd t.
\end{multline*} 
Here,  $H$ (up to a multiplicative constant) is explicitly given by
\[
    |\Delta^2\phi_1|_2^2+\|w_1\|_{\sV}^2 +\left|\dfrac{\delta E}{\delta \phi}(\phi_1)\right|_2^2+\|\phi_1\|_{\sH^2}^8+\|\phi_2\|_{\sH^2}^8+1.
\]
By the conditions \eqref{eq.cond.uniqueness}, the quantity $\int_0^TH(t)\,\dd t$ is bounded.
Then we can apply Gronwall's lemma to deduce 
\[  
     |w(t)|_2^{2}+ \alpha ^{2}|\nabla w(t)|_2^{2} + G(\phi) (t) \leq 0 
\] 
which implies $(w_1,\phi_1)=(w_2,\phi_2)$ on the full measure set defined  in \eqref{eq.cond.uniqueness2}.
\end{proof}


\end{document}